\documentclass[letterpaper]{article}\usepackage{mathtools,amssymb,latexsym,graphics,enumerate}
\usepackage[mathscr]{eucal}
\usepackage{amsmath,amsfonts,amsthm,amssymb}
\usepackage{color}
\usepackage{hyperref}
\usepackage[left=1in,right=1in,top=1in,bottom=1in]{geometry}

\numberwithin{equation}{section}
\allowdisplaybreaks

\newcommand{\be}{\begin{eqnarray*}}
\newcommand{\bel}{\begin{eqnarray}}
\newcommand{\ee}{\end{eqnarray*}}
\newcommand{\eel}{\end{eqnarray}}
\newcommand{\ba}{\begin{aligned}}
\newcommand{\ea}{\end{aligned}}

\newcommand{\na}{\nabla}

\newcommand{\pa}{\partial}

\newtheorem{theorem}{Theorem}
\newtheorem{cor}{Corollary}

\newtheorem{lem}{Lemma}
\newtheorem{pro}{Proposition}

\newtheorem{remark}{Remark}

\numberwithin{remark}{section}
\numberwithin{lem}{section}

\newcommand\R{{\mathbb R}}
\newcommand\T{{\mathbb T}}

\newcommand\eps{\epsilon}

\newcommand{\p}{\ensuremath{\partial}}
\newcommand{\n}{\ensuremath{\nonumber}}

\title{Improved well-posedness for the Triple-Deck \\ and related models 
via concavity}
\date{\today}
\author{David Gerard-Varet \footnote{Universit\'e Paris Cité and Sorbonne Universit\'e, CNRS,   Institut de Math\'ematiques de Jussieu-Paris Rive Gauche (IMJ-PRG)  F-75013, Paris, France, \url{david.gerard-varet@imj-prg.fr}}
   \qquad Sameer Iyer \footnote{ Department of Mathematics, University of California, Davis, Davis, CA 95616, \url{sameer@math.ucdavis.edu} } \qquad Yasunori Maekawa \footnote{Department of Mathematics, Graduate School of Science, Kyoto University, Japan, \url{maekawa.yasunori.3n@kyoto-u.ac.jp}}}
\begin{document}
\maketitle

\begin{abstract} We establish linearized well-posedness of the Triple-Deck system in Gevrey-$\frac32$ regularity in the tangential variable, under concavity assumptions on the background flow. Due to the recent result \cite{DietertGV}, one cannot expect a generic improvement of the result of \cite{IyerVicol} to a weaker regularity class than real analyticity. Our approach exploits two ingredients, through an analysis of space-time modes on the Fourier-Laplace side: i) stability estimates at the vorticity level, that involve the concavity assumption and a subtle iterative scheme adapted from \cite{GVMM}  ii)  smoothing properties of the Benjamin-Ono like equation satisfied by the Triple-Deck flow at infinity. Interestingly, our treatment of the vorticity equation also adapts to the so-called hydrostatic Navier-Stokes equations: we show for this system a similar  Gevrey-$\frac32$ linear well-posedness result for concave data, improving at the linear level the recent work \cite{MR4149066}. 
\end{abstract}

\begin{center}
{\it In memory of Anton\'{i}n Novotn\'y}
\end{center}

\section{Introduction}

In this article we are concerned with the wellposedness properties of the Triple-Deck equations:
\begin{subequations} \label{td:1}
\begin{align}
&\p_t u + u \p_x u + v \p_y u - \p_y^2 u = - \p_x p, \\
&\p_x u + \p_y v = 0, \\
&\p_y p = 0
\end{align}
\end{subequations}
which are supplemented with the boundary conditions 
\begin{subequations} \label{td:2}
\begin{align}
&[u, v]|_{y = 0} = 0, \\
&\lim_{y \rightarrow \infty} (u-y) = A(t, x), \\
&\lim_{x \rightarrow \pm\infty} (u - y) = 0,
\end{align}
\end{subequations}
and an initial datum
\begin{align} \label{id}
u|_{t = 0} = y + u_{init}(x, y). 
\end{align}
The key coupling inherent to the Triple-Deck system is the relation that links $A(t, x)$ to the pressure (the so called ``pressure-displacement" relation):
\begin{align} \label{pdr}
p(t, x) = \frac{1}{\pi} p.v. \int_{\mathbb{R}} \frac{\p_x A(x', t)}{x - x'} dx = |\p_x| A(x, t). 
\end{align}

\vspace{2 mm}

The Triple-Deck equations, \eqref{td:1} -- \eqref{id}, are a refinement of the classical Prandtl system, which arise in the study of the zero-viscosity limit of the Navier-Stokes equations in the vicinity of a boundary. 

Indeed, due to the generic mismatch between the no-slip boundary condition imposed for the Navier-Stokes system (let $\vec{U}_\nu = (u_\nu, v_\nu)$ be the Navier-Stokes velocity field with viscosity equal to $\nu > 0$) and the no-penetration condition imposed for Euler (let $\vec{U}_E = (u_E, v_E)$ be the Euler velocity field), one cannot expect the inviscid limit $\vec{U}_\nu \rightarrow \vec{U}_E$ to hold, at least in sufficiently strong topologies (for instance, $L^\infty$ in the variable normal to the boundary). 

Due to this mismatch, characterizing the inviscid limit typically requires matched asymptotic expansions,  of the type first proposed by Prandtl in 1904:  
\begin{equation}
\begin{aligned} \label{Pr:ans}
\vec{U}^{(\nu)}(t,x,Y) & \approx \big[u_E, v_E\big](t,x,Y), \quad Y \gg \nu^{1/2}, \\
 &  \approx \big[u_P, \sqrt{\nu} v_P\big](t,x, \nu^{-1/2}Y), \quad Y \lesssim \nu^{1/2}. 
\end{aligned}
\end{equation}
The leading order vector-field, $[u_P, \sqrt{\nu} v_P]$ appearing in the expansion above is called the Prandtl boundary layer, and can be shown to obey the following limiting system:
\begin{subequations} \label{sys:pr:1}
\begin{align}
&\p_t u_P + u_P \p_x u_P + v_P \p_y u_P + \p_x p_P - \p_y^2 u_P = - \p_x p_E(t, x, 0), \\
&\p_x u_P + \p_y v_P = 0, \\
&\p_y p_P = 0. 
\end{align}
\end{subequations}
with boundary conditions 
\begin{subequations} \label{sys:pr:2}
\begin{align}
&[u_P, v_P]|_{y = 0} = [0, 0], \\
&\lim_{y \rightarrow \infty} u_P = u_E(t,x,0), 
\end{align}
\end{subequations}
together with an initial datum
\begin{align} \label{sys:pr:3}
u_P|_{t = 0} = u_{P,init}(x, y). 
\end{align}

The Prandtl system, \eqref{sys:pr:1} -- \eqref{sys:pr:3} is classical in fluid dynamics, and has been the source of intense investigation from the mathematical fluid dynamics point of view. As the Prandtl system itself is \textit{not} the main focus of study in this article, we refer to the (non-exhaustive) list of references \cite{MR2601044,MR2952715,MR2849481,MR3284569,MR2975371,MR3429469,MR3385340,MR3464051,MR3461362,MR3925144,LiMaYa,MR4097332,MR4271962}. 

Deriving the Prandtl system from the Navier-Stokes equations requires the formal asymptotic expansion \eqref{Pr:ans}, whihc relies itself implicitly on  a few hypotheses in order to be ``valid" (though, as mentioned above, the mathematical validity has only recently been proven/ disproven). One such hypothesis is the relative smallness of tangential derivatives of $\vec{U}_\nu$ compared to normal derivatives. However, in the vicinity of boundary layer separation, the flow is anticipated to form large tangential gradients which therefore falls outside the regime of the standard Prandtl ansatz, \eqref{Pr:ans}.

To account for this, several reduced models have been derived which incorporate the small tangential scales that are inherently present near the separation point. One famous such model is the Triple-Deck, \eqref{td:1} -- \eqref{id}. This system was introduced by Lighthill, \cite{Lighthill}, Stewartson, \cite{Stewartson}, and several other fluid-dynamicists in the twentieth century. It is useful to keep in mind Figure \ref{tdscalepic} which summarizes the scales used to derive the equations. 
\begin{figure}[h]
\centering
\includegraphics[scale=.4]{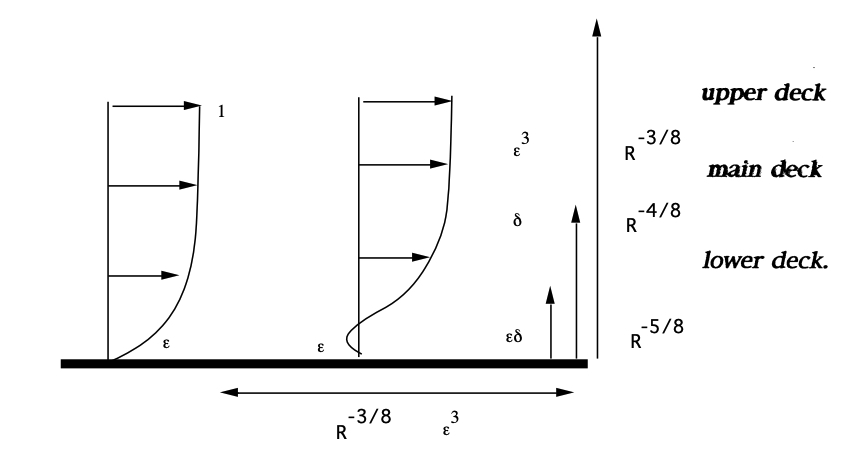}
\caption{Triple-Deck Scales, \cite{Lagree}}
\label{tdscalepic}
\end{figure}

Comparing the Triple-Deck model to the classical Prandtl equation, we see several new mathematical features that are observed to be true near the separation point. Chief among these is \eqref{pdr}, which, physically, represents the velocity at $y = \infty$ entering the fluid domain. Formally, substituting \eqref{pdr} into the momentum equation, \eqref{td:1}, we observe that the momentum equation is forced by $-\p_x |\p_x| A = - \p_x |\p_x| u(t, x, \infty)$, which is a loss of two tangential derivatives and therefore a full derivative too singular to be consistent with even real-analytic wellposedness. Nevertheless, exploiting $L^2$ anti-symmetry of the operator $\p_x |\p_x|$, Iyer-Vicol established in \cite{IyerVicol} that the Triple-Deck system is wellposed in real-analytic spaces. Given this result, a natural question is to weaken the regularity required for wellposedness. However, a recent result of \cite{DietertGV} shows that the Triple-Deck is generically \textit{illposed} in any Gevrey space below real-analyticity, thereby rendering the real-analytic result \cite{IyerVicol} as essentially sharp.  See also \cite{MR3835243} for ill-posedness results in the same spirit. 

As is standard for many wellposedness/illposedness results, the work of \cite{DietertGV} considers the linearized Triple-Deck equations around a smooth shear flow, $V_s = V_s(y) = y + U_s(y)$, which reads 
\begin{subequations} \label{ltd:1}
\begin{align}
&\p_t u + V_s \p_x u + v \p_y V_s - \p_y^2 u = - \p_x |\p_x A|, \label{ltd:1a}\\ 
&\p_x u + \p_y v = 0, 
\end{align}
\end{subequations}
and with boundary and initial conditions 
\begin{subequations} \label{ltd:2}
\begin{align}
 [u,v]\vert_{y=0} & = 0, \\
\lim_{y \rightarrow \infty} u & = A(t, x), \\
 \lim_{x \rightarrow \pm\infty} u & = 0, \\
u\vert_{t=0} & = u_{init}.
\end{align}
\end{subequations}
In this paper, we show that under concavity assumptions on the  shear flow $U_s(y)$, the illposedness mechanism from \cite{DietertGV} no longer holds, and in fact the result of \cite{IyerVicol} can be improved to Gevrey-$\frac32$. In particular, we will assume the following on our background shear flow: $U_s$ is bounded, in $C^3(\overline{\mathbb{R}_+})$, and 
\begin{subequations} \label{assume:Us}
\begin{align}  \label{assume:Us:1}
&U_s'' < 0, \\ \label{assume:Us:2}
&\sup_{y} \langle y \rangle^{6} |U_s''| < \infty, \\ \label{assume:Us:3}
&\sup_y | \frac{U_s'''}{U_s''} | < \infty, \\ \label{assume:Us:4}
&U_s(0) = 0
\end{align}
\end{subequations}
\begin{remark} It will be convenient to choose the normalization 
\begin{align} \label{normalizeinfty}
U_s(\infty) = 1, 
\end{align}
though this is simply to alleviate some notation and remove factors of $U_s(\infty)$ appearing in the analysis. 
\end{remark}

Our main result is as follows.
\begin{theorem} \label{thm:main} Assume the shear flow $V_s(y) = y + U_s(y)$ is given such that $U_s$ satisfies \eqref{assume:Us}-\eqref{normalizeinfty}. Assume that the initial data has Gevrey $3/2$ regularity in $x$ and Sobolev regularity in $y$, namely that for some $c_0 > 0$, 
$$  \| e^{c_0 |\pa_x|^{2/3}}(1+y)^3  \pa_y u_{init}\|_{L^2(\R \times (0,1))}  < +\infty $$
together with the Dirichlet condition $u_{init}\vert_{y=0} = 0$.  Then, system \eqref{ltd:1}-\eqref{ltd:2} has a unique local in time solution  that obeys the following estimate, for some constants $\beta, C, s > 0$, and for all time $t < \frac{c_0}{\beta}$: 
\begin{equation}
\|  e^{(c_0 - \beta t) |\pa_x|^{2/3}} (1+y)^2 \pa_y u(t, \cdot)\|_{L^2(\R \times (0,1))} \le C  \| e^{c_0 |\pa_x|^{2/3}} (1+|\pa_x|)^s  (1+y)^3 \pa_y u_{init}\|_{L^2(\R \times (0,1))} 
\end{equation}
\end{theorem}
The proof will be outlined in the next section. It relies on delicate estimates on the couple $(\omega, A)$, where $\omega = \pa_y u$ is an analogue of the vorticity adapted to this anistropic model.  One key aspect is the derivation of estimates on $\omega$, given $A$. This is where the concavity assumption plays a role. The stabilizing effect of concavity in inviscid flows has been well-known since the pioneering works of Lord Rayleigh, and has been exploited in the proof of various mathematical stability results \cite{MR1690189,MR1761422,MR3385340}. But associated mathematical techniques do not easily transfer to viscous flows, due to vorticity creation at the boundary, which is a potential source of other instabilities. To overcome this problem in the context of the linearized Triple-Deck model, we derive estimates using an iterative scheme inspired from the work \cite{GVMM} on stability of Prandtl solutions of the Navier-Stokes equations. The main ideas behind this scheme are explained in Paragraph \ref{subsec:iteration}. Its convergence requires Gevrey $3/2$ regularity. Once the estimate for $\omega$ is obtained, we turn to the horizontal velocity at infinity $A = A(t,x)$. Here, we make a crucial use of the Benjamin-Ono like equation satisfied by $A$. Distinguishing between several regions of the spectral plane $(\lambda,k)$ (after Laplace transform in time, Fourier transform in $x$), we manage to obtain good resolvent estimates for the linearized Triple-Deck model, from which Gevrey stability estimates follow. 

Our scheme for the derivation of vorticity estimates  has applications beyond the Triple-Deck model. It notably applies to the study of the so-called hydrostatic Navier-Stokes equation, which stems from the analysis of the usual  Navier-Stokes equation in a narrow channel of width $\eps$:
\begin{equation*}
\begin{aligned}
\pa_t \vec{u}  + \vec{u} \cdot \na \vec{u}   - \nu \Delta \vec{u}  +  \na p & = 0, \quad x \in \R, z \in (0,\eps), \\
\nabla \cdot \vec{u} & = 0,  \quad  x \in \T, z \in (0,\eps), \\
\vec{u}\vert_{z=0,\eps} & = 0. 
\end{aligned}
\end{equation*}
In the case where $\nu \sim \eps^2$, approximation 
$$ \vec{u} \sim  \Big( u(t,x,z/\eps), \eps v(t,x,z/\eps) \Big) $$
yields to the reduced model: 
 \begin{equation} \label{HNS} 
\begin{aligned}
\pa_t u   + u \pa_x u + v \pa_y u - \pa^2_y u +   \pa_x p & = 0, \quad x \in \R, y \in (0,1), \\
\pa_y p  & = 0,  \quad x \in \R, y \in (0,1), \\
\pa_x u + \pa_y v  & = 0,  \quad x \in \T, y \in (0,1), \\
u\vert_{y=0,1} = v\vert_{y=0,1} & = 0. 
\end{aligned}
\end{equation}
This model shares features with the Triple-Deck model. For general data, one can not expect more than local analytic well-posedness, due to a strong inviscid  instability mechanism identified in \cite{Renardy09}. On the contrary, under concavity (or convexity) of the initial data, it is known that  the hydrostatic Euler equation is well-posed in Sobolev regularity \cite{MR2898740}. It is then natural to ask what remains of this improved stability in the presence of diffusion, namely for system \eqref{HNS}. 
A partial answer was brought very recently in \cite{MR4149066}, where a local well-posedness result was achieved in Gevrey regularity, but with an exponent $8/7$ close to $1$ for technical reasons. See also \cite{MR4125518,LiLiTi} for related works. 

It turns out that adapting the methodology of the present paper, one can improve such result at the linear level, by establishing Gevrey $3/2$ well-posedness for the problem: 
\begin{equation} \label{LHNS} 
\begin{aligned}
\pa_t u   + U_s \pa_x u + v U_s' - \pa^2_y u +   \pa_x p & = 0, \quad x \in \R, y \in (0,1), \\
\pa_y p  & = 0,  \quad x \in \R, y \in (0,1), \\
\pa_x u + \pa_y v  & = 0,  \quad x \in \R, y \in (0,1), \\
u\vert_{y=0,1} = v\vert_{y=0,1} & = 0. 
\end{aligned}
\end{equation}
where $U_s = U_s(y)$ is again a concave shear flow (convex shear flow would work as well). Namely, we have the following result, very similar to Theorem \ref{thm:main}: 
\begin{theorem} \label{thm:main2} Assume that  the shear flow $U_s$ is smooth and strictly concave on $[0,1]$. Assume that the initial data has Gevrey $3/2$ regularity in $x$ and Sobolev regularity in $y$, namely that for some $c_0 > 0$, 
$$  \| e^{c_0 |\pa_x|^{2/3}}  \pa_y u _{init}\|_{L^2(\R \times \R_+)}  < +\infty $$
together with the  conditions $u_{init}\vert_{y=0,1} = 0$, $\int_0^1 \pa_x u_{init} dy = 0$.  Then, system \eqref{LHNS} has a unique  local in time solution with initial data $u_{init}$,  that obeys the following estimate, for some constants $\beta, C, s > 0$, and for all time $t < \frac{c_0}{\beta}$: 
\begin{equation}
\|  e^{(c_0 - \beta t) |\pa_x|^{2/3}}  \pa_y u(t, \cdot)\|_{L^2(\R \times \R_+)} \le C  \| e^{c_0 |\pa_x|^{2/3}} (1+|\pa_x|)^s  \pa_y u_{init}\|_{L^2(\R \times \R_+)} 
\end{equation}
\end{theorem}
The proof of this theorem will be quickly explained in the last Section \ref{sec:HNS}. It is very similar in  spirit with the analysis carried for the Triple-Deck, without difficulties coming from the coupling with the unknown $A$.

\section{Outline of the proof}
We explain here the main steps in the proof of Theorem \ref{thm:main}. Due to the several averaging operators we have in our analysis, we will introduce the following notations:
\begin{align} \label{defn:cal:op}
&\mathcal{U}[\omega] := \int_0^\infty \omega, \qquad \mathcal{U}_y[\omega] := \int_y^{\infty} \omega, \\
 & \mathcal{V}[\omega] := \int_0^{\infty} \int_{\infty}^y \omega, \qquad  \mathcal{V}_y[\omega] := \int_0^y \int_{\infty}^{y'} \omega, \qquad \mathcal{V}^{(-)}_y[ \omega] := \int_y^{\infty} \int_{y'}^{\infty} \omega. 
\end{align}

\subsection{$(\omega, A)$ Decomposition} \label{subsec:decompo}

Our starting point is to decompose \eqref{ltd:1} into two coupled equations: one governing the vorticity, $\omega := \p_y u$, and one governing the trace at infinity $A(t, x)$. To derive the evolution governing $A$, we evaluate \eqref{ltd:1} at $y = \infty$, whereas to obtain the evolution governing $\omega$, we differentiate \eqref{ltd:1} with respect to $y$. In both cases, we use the identity  
\begin{align} \label{v:id:1}
v = -\int_0^y u_x = - \int_0^y  (A_x - \int_{y'}^{\infty} \omega_x) = - y A_x + \int_0^y \int_{y'}^{\infty} \omega_x = - y A_x - \mathcal{V}_y[\omega_x].
\end{align}
We thus obtain the vorticity-Benjamin-Ono equation 
\begin{subequations} \label{vbo}
\begin{align}
\label{vboa} &\p_t A + \p_x A + \p_x |\p_x| A = \mathcal{V}[\omega_x], \\ 
\label{vbob} &\p_t \omega + V_s \p_x \omega - U_s'' \mathcal{V}_y[\omega_x] - \p_y^2 \omega = y U_s'' A_x, \\
\label{vboc} &\mathcal{U}[\omega] = A. 
\end{align}
\end{subequations}
Conversely, starting from a solution $(\omega,A)$ of \eqref{vbo} with $(\omega, A)\vert_{t=0} = (\pa_y u_{init}, u_{init}(x,\infty))$, it is easy to check that the triplet $(u,v,A)$, where $u := \int_0^y \omega$, $v := - \int_0^y \pa_x u$ satisfies \eqref{ltd:1}-\eqref{ltd:2}. We insert here a proof of this fact.
\begin{lem} \label{lem:u:int}Assume the tuple $(\omega, A)$ satisfies \eqref{vbo}. Then $u := \int_0^y \omega$ satisfies \eqref{ltd:1} -- \eqref{ltd:2}. 
\end{lem}
\begin{proof} The velocity $u := \int_0^y \omega$ automatically satisfies $u|_{y = 0} = 0$. By using the identity 
\begin{align*}
\p_y \{\p_t u + V_s \p_x u + v V_s' - \p_y^2 u \} =  \p_t \omega + V_s \p_x \omega + v V_s'' - \p_y^2 \omega, 
\end{align*}
coupled with the identity \eqref{v:id:1}, the equations \eqref{vbob} and \eqref{vboc} implies 
\begin{subequations}
\begin{align} \label{u:int:eq:1}
&\p_t u + V_s \p_x u - (y A_x + \mathcal{V}_y[\omega_x]) V_s' - \p_y^2 u = C(t, x), \\ \label{u:int:eq:2}
&u|_{y = 0} = 0, \qquad u|_{y = \infty} = A
\end{align}
\end{subequations}
for an undetermined function $C(t, x)$. To determine $C(t, x)$, we evaluate \eqref{u:int:eq:1} at $y = \infty$, which produces
\begin{align}
\p_t A + \p_x A - \mathcal{V}[\omega_x] = C(t, x)
\end{align}
Upon invoking \eqref{vboa}, we deduce $C(t, x) = - \p_x |\p_x|A$. Inserting into \eqref{u:int:eq:1}, we obtain \eqref{ltd:1a}. 
\end{proof}

In particular, evaluating equation \eqref{ltd:1a} at $y=0$, we find 
\begin{equation} \label{neumann_omega}
\pa_y \omega\vert_{y=0} = \pa_x |\pa_x| A 
\end{equation}
One can further remark that \eqref{vboa}-\eqref{vbob}-\eqref{neumann_omega} is equivalent to \eqref{vbo}. Indeed,  we have just seen that \eqref{vbo} implies the Neumann condition \eqref{neumann_omega}. Conversely, if $(\omega,A)$ satisfies  \eqref{vboa}-\eqref{vbob}-\eqref{neumann_omega}, then, still defining 
$u := \int_0^y \omega$, we obtain easily  \eqref{ltd:1a}, and evaluating this equation at $y=\infty$: 
$$ \p_t \mathcal{U}[\omega] + \p_x \mathcal{U}[\omega] + \p_x |\p_x| A = \mathcal{V}[\omega_x] $$
so that combining with  \eqref{vboa}, we find
$\p_t (\mathcal{U}[\omega]  - A) + \p_x ( \mathcal{U}[\omega]  - A)  = 0$. This implies \eqref{vboc} thanks to the compatibility condition on the initial data for $(\omega,A)$.

We now take formally the Laplace transform in time and Fourier transform in $x$ of system \eqref{vbo}. We find 
\begin{subequations} 
\begin{align} \label{eq:vort:a}
&(\lambda + i k + i k |k|) \hat{A}  = i k \mathcal{V}[ \hat{\omega}] + \hat{A}_{init}, \\ \label{eq:vort:b}
& (\lambda + i k V_s)  \hat{\omega} - ik U''_s \mathcal{V}_y[ \hat{\omega}] - \pa^2_y \hat{\omega} = i k U''_s(y) y  \hat{A} + \hat{\omega}_{init}, \\ \label{eq:vort:c}
& \mathcal{U}[\hat{\omega}]  = \hat{A}, 
\end{align}
\end{subequations}
involving 
\begin{align*}
 (\hat{\omega} , \hat{A}) & =   (\hat{\omega}(y), \hat{A}) :=  \mathcal{L}_{t \rightarrow \lambda} \mathcal{F}_{x \rightarrow k} (\omega(\cdot,y), A ), \\  
(\hat{\omega}_{init}, \hat{A}_{init}) & = (\hat{\omega}_{init}(y), A_{init}) := \mathcal{F}_{x \rightarrow k} (\omega_{init}, A_{init}) 
\end{align*}
We denote
\begin{equation} \label{defi:H}
 H := \Big\{ (\hat{\omega},\hat{A}) \in L^2(\R_+, (1+y)^3 dy) \times \R,  \Big\}, 
 \end{equation}
 equipped with the norm 
$$  \|(\hat{\omega},\hat{A})\|_H = \Big( \|(1+y)^3 \hat{\omega}\|_{L^2(\R_+)}^2 + |\hat{A}|^2 \Big)^{1/2}. $$
Our key result will be the following: 
 \begin{pro} \label{pro:main}
 There exists absolute positive constants $K_{\ast},  k_0$ and $M$, such that for all $|k| \ge k_0$,  all $\lambda$ with $\Re(\lambda) \ge K_{\ast} |k|^{2/3}$, and all data $(\hat{\omega}_{init}, \hat{A}_{init}) \in H$,  system \eqref{eq:vort:a}-\eqref{eq:vort:b}-\eqref{eq:vort:c} has a unique solution satisfying 
$$  \| (\hat{\omega} , \hat{A})\|_H  \lesssim |k|^{1/3} |\lambda|^{1/4} \| (\hat{\omega}_{init} , \hat{A}_{init})\|_H . $$
 \end{pro}
Most of the  paper is devoted to the proof of this proposition.  The main steps of this proof, involving an iteration scheme inspired from \cite{GVMM},  will be given in the next two paragraphs.  Well-posedness of the systems involved at each step of the iteration is shown in Sections \ref{sec:OmegaBL} and \ref{sec:OmegaH}.  Checking the convergence of the iteration is  done in Section \ref{sec:convergence}. Eventually, we will explain in Section \ref{sec:proof:thm} how to complete Proposition \ref{pro:main} to obtain Theorem \ref{thm:main}.

\subsection{Hydrostatic \& Boundary Layer Iteration} \label{subsec:iteration}

Here, as well as in Sections \ref{sec:OmegaBL}, \ref{sec:OmegaH} and \ref{sec:convergence}, we will focus on the system \eqref{eq:vort:a}-\eqref{eq:vort:b}-\eqref{eq:vort:c}.   For a significant part of the analysis, we view $\hat{A}$ as arbitrary and given, and we study in isolation the system \eqref{eq:vort:b} -- \eqref{eq:vort:c} for $\hat{\omega}$. To emphasize this point of view, we will write $\hat{\omega} = \hat{\omega}\big[\hat{A}\big]$. We split 
\begin{align} \label{omega:structure:1}
\hat{\omega}\big[\hat{A}\big] = \hat{A} \, \bar{\omega} +  \omega_{inhom},
\end{align}
 where 
\begin{equation} \label{eq:vort:bar}
\begin{aligned}
& (\lambda + i k V_s) \bar{\omega}   - ik U''_s \mathcal{V}_y[\bar{\omega}]  - \pa^2_y \bar{\omega}  = i k U''_s(y) y, \\
& \mathcal{U}[ \bar{\omega} ]  = 1, 
\end{aligned}
\end{equation}
and 
\begin{equation} \label{eq:vort:inh}
\begin{aligned}
& (\lambda + i k V_s) \omega_{inhom}  - ik U''_s \mathcal{V}_y[ \omega_{inhom}] - \pa^2_y \omega_{inhom} =  \hat{\omega}_{init}, \\
&\mathcal{U}[ \omega_{inhom}]  = 0.
\end{aligned}
\end{equation}

\subsubsection{Iteration for $\overline{\omega}$}

In this section, we explain the main strategy for the construction of  the normalized quantity $\bar{\omega}$, solution of \eqref{eq:vort:bar}. This construction will work for $\Re(\lambda)$ large enough, that is for  $\Re(\lambda) \ge K_\ast |k|^{2/3}$, $|k| \ge k_0$, for some absolute constants $K_\ast, k_0$.  We plan to construct $\bar{\omega}$ in two pieces. The first one  accounts for the (singular in $k$) forcing $ik U_s''(y) y$, but has a homogeneous Neumann condition instead of mean $1$. It is called the hydrostatic part and  denoted $\omega_H$, as stability estimates for this part take their inspiration from works on hydrostatic Euler: see \cite{MR1690189,MR2898740}, as well as \cite{MR1761422}.  The second piece  has (essentially) no forcing, but  corrects the mean. This piece is  called the boundary layer part, and denoted $\omega_{BL}$. Actually, as will be seen below, we will not be able to correct the mean condition at once without creating some error source term. This will require in turn to add an hydrostatic term, which will create  an error in the mean, and so on. Hence, both the hydrostatic and boundary layer parts will be given as infinite sums. This idea of solving a fluid equation through an iteration has revealed fruitful in several recent papers, notably around the analysis of Orr-Sommerfeld equations: see the pioneering work \cite{GGN2014}, as well as \cite{MR3855356}.  Our main source of inspiration here is \cite{GVMM}. 
 More precisely, the idea is to  construct $\bar{\omega}$ under the form:  
\begin{align} 
\overline{\omega} = & \omega_{H} + \omega_{BL} \\ \n
= & \sum_{j = 0}^{\infty}  \omega_{H}^{(j)}  + \sum_{j = 0}^{\infty} \omega_{BL}^{(j)} \\ \n
= & \omega_H^{(0)}  + \omega_{BL}^{(0)} + \sum_{j = 1}^{\infty} \omega_H^{(j)} + \sum_{j = 1}^{\infty} \omega_{BL}^{(j)} \\ \label{exp:1}
= & \omega_{H}^{(0)} + \omega_{BL}^{(0)} + \omega_{H}^{(tail)} + \omega_{BL}^{(tail)}. 
\end{align}
where  the ``tail" of both the expansions will be shown   to be higher order. We delineate here the various systems satisfied formally by  $\omega_{H}^{(j)}$ and $ \omega_{BL}^{(j)}$, $j \ge 0$. 

The first idea is to initialize the construction by solving the Neumann problem: 
\begin{equation} \label{eq:vort:bar:a}
\begin{aligned}
& (\lambda + i k V_s) \omega_H^{(0)}   - ik U''_s \mathcal{V}_y[ \omega_H^{(0)}]   - \pa^2_y \omega_H^{(0)}  = i k U''_s(y) y, \\
& \p_y \omega_H^{(0)}|_{y = 0} = 0. 
\end{aligned}
\end{equation}
This system will be shown to be well-posed for $\Re(\lambda)$ large enough in Section \ref{sec:OmegaH}.  

We then initialize the boundary layer construction by solving the system: 
\begin{equation} \label{eq:vort:bar:BL:A}
\begin{aligned}
& (\lambda + i k V_s) \omega_{BL}^{(0)}   - \pa^2_y \omega_{BL}^{(0)}  = 0, \\
& \mathcal{U}[ \omega_{BL}^{(0)}] = 1 - \mathcal{U}[ \omega_H^{(0)}], 
\end{aligned}
\end{equation}
Well-posedness of this system will be shown in Section \ref{sec:OmegaBL}. Note that we got rid at this step of the stretching term.  This creates an error term  
$- i k U_s'' \mathcal{V}_y[ \omega_{BL}^{(0)} ]$, which will be corrected by the next hydrostatic term in the expansion. 
With this in mind, we  define our $j$'th order hydrostatic term in the expansion \eqref{exp:1} $(j \ge 1$) as the solution of 
\begin{equation} \label{eq:vort:bar:aj}
\begin{aligned}
& (\lambda + i k V_s) \omega_H^{(j)}   - ik U''_s \mathcal{V}_y[ \omega_H^{(j)}]   - \pa^2_y \omega_H^{(j)}   =   i k U_s'' \mathcal{V}_y[ \omega_{BL}^{(j-1)} ] \\
& \p_y \omega_H^{(j)}|_{y = 0} = 0, 
\end{aligned}
\end{equation}
and for the $j$'th order boundary layer term $(j \ge 1)$: 
\begin{equation} \label{eq:vort:bar:BL:Aj}
\begin{aligned}
& (\lambda + ik V_s)\omega_{BL}^{(j)}  - \pa^2_y \omega_{BL}^{(j)}  = 0, \\
&\mathcal{U}[ \omega_{BL}^{(j)}] =  - \mathcal{U}[ \omega_H^{(j)}], 
\end{aligned}
\end{equation}
We stress that this sequence of profiles $(\omega_H^{(j)},\omega_{BL}^{(j)})$ can be expressed in terms of a sequence of parameters  and of two  fixed functions. To see this, we introduce
\begin{align} \label{hg1}
\Lambda_j :=  &-\mathcal{U}[\omega_H^{(j)}]  \\ \label{def:L00}
\lambda_j := & \begin{cases} 1 +  \Lambda_0, \qquad j = 0 \\ \Lambda_j, \qquad j \ge 1 \end{cases}
\end{align}
as well as the solution $\Omega_{BL}$ of  
\begin{equation} \label{gtgt2}
\begin{aligned}
& (\lambda + ik V_s)\Omega_{BL}  - \pa_y^2 \Omega_{BL}  = 0, \\
& \mathcal{U}[ \Omega_{BL}] =1.
\end{aligned}
\end{equation}
and the solution $F_H$ of 
\begin{equation} \label{gtgt}
\begin{aligned}
& (\lambda + i k V_s) F_H   - ik U''_s \mathcal{V}_y[ F_H]   - \pa^2_y F_H  =  U_s'' \mathcal{V}_y[\Omega_{BL}], \\
& \p_y F_H|_{y = 0} = 0, 
\end{aligned}
\end{equation}
These two systems will be shown to have solutions for $\Re(\lambda) \ge K_\ast |k|^{2/3}$ in Sections \ref{sec:OmegaBL} and \ref{sec:OmegaH}.   Clearly, it follows that 
\begin{align} \label{hg3}
\omega^{(j)}_{BL} = &\lambda_j \Omega_{BL}, \qquad j \ge 0, \\ \label{hg4}
\omega_H^{(j)} = & ik \lambda_{j-1} F_H, \qquad j \ge 1.
\end{align}
and inserting this relation into the  formulae for $\lambda_j$, we find 
\begin{align} \label{hg5}
 \lambda_{j+1} & = (- i k\mathcal{U}[F_H]) \lambda_j, \quad \text{that is } \: \lambda_j = (- i k\mathcal{U}[F_H])^j \lambda_0, \quad j \ge 0. 
 \end{align}
From these relations, we see that all profiles $(\omega_H^{(j)},\omega_{BL}^{(j)})$ only depend on $\omega_H^{(0)}$, $F_H$ and $\Omega_{BL}$. Moreover, as $\lambda_j$ obeys a geometric progression, the convergence of the series will depend on whether the common ratio $(- i k\mathcal{U}[F_H])$  is less  than $1$, which will be examined in Section \ref{sec:convergence}. 


\subsubsection{Iteration for $\omega_{inhom}$}

We are led to perform a similar iteration to construct $\omega_{inhom}$, the solution to \eqref{eq:vort:inh}. We decompose as follows: 
\begin{align} \n
\omega_{inhom} = & \omega_{IH} + \omega_{IB} \\ \n
= & \omega_{IH}^{(0)} + \omega_{IB}^{(0)} + \sum_{j = 1}^{\infty} \omega_{IH}^{(j)} + \sum_{j = 1}^{\infty} \omega_{IB}^{(j)} \\ \label{abcded}
= & \omega_{IH}^{(0)} + \omega_{IB}^{(0)} + \omega_{IH}^{(tail)}  + \omega_{IB}^{(tail)}
\end{align}
We will now describe the quantities appearing above. 

At leading order, similarly to the previous paragraph, we want $\omega^{(0)}_{IH}$  to solve
\begin{equation} \label{eq:vort:IH}
\begin{aligned}
& (\lambda + i k V_s) \omega^{(0)}_{IH}  - ik U''_s \mathcal{V}_y[ \omega^{(0)}_{IH}] - \pa^2_y \omega^{(0)}_{IH} =  \omega_{init}, \\
& \p_y \omega_{IH}|_{y = 0}  = 0,
\end{aligned}
\end{equation}
(see Section \ref{sec:OmegaH} for well-posedness) and  $\omega^{(0)}_{IB}$ to solve
\begin{equation} \label{eq:vort:bar:BL:ABAB}
\begin{aligned}
& (\lambda + i k V_s) \omega_{IB}^{(0)}   - \pa^2_y \omega_{IB}^{(0)}  = 0, \\
& \mathcal{U}[ \omega_{IB}^{(0)}] =  - \mathcal{U}[ \omega_{IH}^{(0)}].
\end{aligned}
\end{equation}
Again,  this construction of $\omega_{IB}^{(0)}$ creates an error $- i k U''_s \mathcal{V}_y[\omega_{IB}^{(0)}]$.   

We now define the higher order ``tail" quantities. We define for $j \ge 1$,
\begin{equation} \label{cdefIH}
\begin{aligned}
& (\lambda + i k V_s) \omega^{(j)}_{IH}  - ik U''_s \mathcal{V}_y[ \omega^{(j)}_{IH}] - \pa^2_y \omega^{(j)}_{IH} =   i k U''_s \mathcal{V}_y[\omega_{IB}^{(j-1)}], \qquad j \ge 1 \\
& \p_y \omega^{(j)}_{IH}|_{y = 0}  = 0,
\end{aligned}
\end{equation}
and
\begin{equation} \label{uyht}
\begin{aligned}
& (\lambda + i k V_s) \omega_{IB}^{(j)}   - \pa^2_y \omega_{IB}^{(j)}  = 0, \\
& \mathcal{U}[ \omega_{IB}^{(j)}] =  - \mathcal{U}[ \omega_{IH}^{(j)}].
\end{aligned}
\end{equation}
We once again propose 
\begin{align}
\widetilde{\lambda}_j := - \mathcal{U}[\omega^{(j)}_{IH}],  \qquad j \ge 0.
\end{align}
Given these definitions, we have 
\begin{align}
\omega^{(j)}_{IB} = & \widetilde{\lambda}_j \Omega_{BL}, \qquad j \ge 0 \\
 \omega^{(j)}_{IH} = & ik \widetilde{\lambda}_{j-1} F_H, \qquad j \ge 1
\end{align}
where again, 
\begin{equation} \label{tildelambdageometric}
\tilde{\lambda}_{j+1} = (-i k \mathcal{U}[F_H]) \tilde{\lambda}_j, \quad \text{that is } \: \widetilde{\lambda}_j = (-i k \mathcal{U}[F_H])^j \tilde{\lambda}_0, \quad j \ge 0.
\end{equation}

\subsection{Gevrey Stability Estimate}

We will be working under the hypotheses 
\begin{align} \label{k:hyp}
|k| \ge & k_0, \\ \label{lambda:hyp}
\Re(\lambda) \ge & K_{\ast} |k|^{\frac23}.   
\end{align}
where $K_{\ast}, k_0 >> 1$ relative to universal constants. We introduce the following weight, which will be used throughout our analysis:
\begin{equation} \label{Usk}
 U_{s,k}''  = U''_s - |k|^{-2/3} (1+y)^{-6}
 \end{equation}
We first state the following elementary properties of the weight $U_{s,k}''$:
\begin{lem} The weight \eqref{Usk} satisfies the following upper and lower bounds:
\begin{align} \label{lowerupper1}
(1 + y)^6 \lesssim \frac{1}{-U_{s,k}''} \lesssim |k|^{\frac23} (1 + y)^6
\end{align}
\end{lem}
\begin{proof} For the upper bound, we have 
\begin{align*}
\frac{1}{-U_{s,k}''} = \frac{1}{- U_s'' + |k|^{-\frac23}(1 + y)^{-6}} \le \frac{1}{ |k|^{-\frac23}(1 + y)^{-6}} = |k|^{\frac23} (1 + y)^{6}.
\end{align*}
For the lower bound, we have the general elementary inequality for $a, b \ge 0$:
\begin{align*}
\frac{1}{a+b} \ge \frac{1}{2 \max \{a, b \}}. 
\end{align*}
Given this, the lower bound will follow from the following upper bound 
\begin{align*}
\max\{- U_s'', |k|^{-\frac{2}{3}}(1 + y)^{-6} \} \lesssim (1+y)^{-6}, 
\end{align*}
upon invoking our decay assumption, \eqref{assume:Us:2}. 
\end{proof}

As a result of the formal analysis of the two previous paragraphs, and of the rigorous analysis of Sections \ref{sec:OmegaBL} to \ref{sec:convergence}, we  state our main proposition on the structure of  $\hat{\omega}\big[\hat{A}\big]$ from \eqref{omega:structure:1}. 
\begin{pro} \label{propn:s} Under  \eqref{k:hyp}-\eqref{lambda:hyp}, for any constant $\hat{A}$, there is a unique solution $\hat{\omega}\big[\hat{A}\big]$ of \eqref{eq:vort:b}-\eqref{eq:vort:c} that can be decomposed in the following manner: 
\begin{align} \label{structure:cond:om}
\hat{\omega}\big[\hat{A}\big] = \hat{A} \lambda_{\ast} \Omega_{BL} + \hat{A} \omega_H^{(0)} + \hat{A}  \omega_H^{(tail)} + \omega_{inhom},
\end{align}
where  $\lambda_{\ast} \in \mathbb{C}$,  and where the functions $\Omega_{BL}, \omega_H^{(0)},  \omega_H^{(tail)}, \omega_{inhom} \in L^2((1+y)^3 \R_+)$  were introduced in the previous paragraph. Moreover, they satisfy the following bounds: 
\begin{subequations}
\begin{align} \label{s:a}
|\lambda_{\ast}| \lesssim &  1 +  \frac{|k|}{\Re(\lambda)} \\ \label{s:b}
\mathcal{V}[|\Omega_{BL}|]  \lesssim & \frac{1}{|\lambda|^{\frac12}}, \\ \label{s:c}
\mathcal{V}[|\omega_H^{(0)}| ] \lesssim & \frac{|k|}{\Re(\lambda)} \\  \label{s:d}
\mathcal{V}[|\omega_H^{(tail)}|] \lesssim & \frac{|k|}{\Re(\lambda) |\lambda|^{\frac12}} \Big(1 + \frac{|k|}{\Re(\lambda)}\Big), \\  \label{s:e}
\mathcal{V}[|\omega_{inhom}|] \lesssim & \frac{1}{\Re(\lambda)} |k|^{1/3} \| (1+y)^3 \omega_{init} \|_{L^2_y}.
\end{align}
\end{subequations}
\end{pro}
This proposition will be proven in Section \ref{sec:convergence}. Given this structural decomposition of $\hat{\omega} = \omega\big[\hat{A}\big]$, \eqref{structure:cond:om}, we can prove 
\begin{pro} Under  \eqref{k:hyp}-\eqref{lambda:hyp}, the equation \eqref{eq:vort:a}, where $\hat{\omega} = \hat{\omega}\big[\hat{A}\big]$ was  introduced in Proposition \ref{propn:s} has a unique solution $\hat{A}$ satisfying: 
\begin{align} \label{estimA}
|\hat{A}|^2 \lesssim  \| (1+y)^3 \omega_{init} \|_{L^2_y}^2 + \frac{1}{|k|^{\frac43}} |\hat{A} _{init}|^2.
\end{align}
\end{pro}
\begin{proof} For brevity, we focus on the {\it a priori} estimate.  We rewrite \eqref{eq:vort:a} upon recalling the definition of $\mathcal{V}$ in \eqref{defn:cal:op} and invoking the structural decomposition \eqref{structure:cond:om} as 
\begin{align} \n
(\lambda + i k + ik|k|) \hat{A} = & i k \mathcal{V}[\hat{\omega}] + \hat{A} _{init} \\ \label{got1}
= &  i k\lambda_{\ast} \hat{A} \mathcal{V}[ \Omega_{BL}  ] + ik \hat{A} \mathcal{V}[ \omega_H^{(0)}] + ik \hat{A} \mathcal{V}[ \omega_H^{(tail)}] + ik \mathcal{V}[\omega_{inhom}] + \hat{A} _{init}. 
\end{align}
We distinguish between two regimes in the $(\lambda, k)$ space. 

\vspace{2 mm}

\noindent \textit{Case 1: $|\lambda + i k |k| | \ge \frac12 |k|^2$} In this case, we can simply divide both sides of \eqref{got1} by $|\lambda + i k |k||$, and take the modulus. We obtain as a result 
\begin{align} \n
|\hat{A}| \le & \Big( \frac{|k| \lambda_{\ast}}{|\lambda + i k |k| |} |\mathcal{V}[ \Omega_{BL}]| + \frac{|k|}{|\lambda + i k |k||} | \mathcal{V}[ \omega_H^{(0)}]|  + \frac{|k|}{|\lambda + i k |k||} | \mathcal{V}[ \omega_H^{(tail)}]|   \Big) | \hat{A} | \\ \n
& + \frac{ |k|}{|\lambda + i k |k|| } | \mathcal{V}[\omega_{inhom}]| + \frac{1}{|\lambda + i k |k||} |\hat{A}_{init}| \\ \n
\lesssim & \Big( \frac{\lambda_{\ast}}{|k|} |\mathcal{V}[ \Omega_{BL}]| + \frac{1}{|k|} | \mathcal{V}[ \omega_H^{(0)}]|  + \frac{1}{|k|} | \mathcal{V}[ \omega_H^{(tail)}]|   \Big) | \hat{A} | + \frac{1}{|k|} | \mathcal{V}[\omega_{inhom}]| + \frac{1}{k^2} |\hat{A}_{init}| \\ \n
\lesssim & \underbrace{\Big( \frac{1}{|k| |\lambda|^{1/2}}  \Big(1 + \frac{|k|}{\Re(\lambda)}\Big)   + \frac{1}{\Re(\lambda)}  + \frac{1}{\Re(\lambda) |\lambda|^{1/2}} \Big(1 + \frac{|k|}{\Re(\lambda)}\Big)  \Big)}_{\bold{m}(\lambda, k)}  | \hat{A} | + \frac{1}{|k|} | \mathcal{V}[\omega_{inhom}]| + \frac{1}{k^2} |\hat{A}_{init}| \\ \label{ca1}
\end{align}
Upon invoking \eqref{k:hyp} and \eqref{lambda:hyp}, we bound the Fourier-Laplace multiplier, $\bold{m}(\lambda, k)$, appearing above via 
\begin{align} \n
| \bold{m}(\lambda, k)|   
\lesssim &  \frac{1}{|k| |\lambda|^{1/2}} +  \frac{1}{\Re(\lambda) |\lambda|^{1/2}}  +  \frac{1}{\Re(\lambda)} +   \frac{|k|}{\Re(\lambda)^2 |\lambda|^{1/2}} \\
\lesssim &  \frac{1}{|k| |\lambda|^{1/2}} +  \frac{1}{\Re(\lambda)}  \lesssim \frac{1}{|k|^{2/3}}. 
\end{align}
Inserting back into \eqref{ca1}, we obtain 
\begin{align} \n
|\hat{A}| \lesssim & \frac{1}{|k|^{2/3}}  | \hat{A} | + \frac{ 1}{|k|} | \mathcal{V}[\omega_{inhom}]| + \frac{1}{k^2} |\hat{A}_{init}|,
\end{align}
which closes the estimate for $\hat{A}$, and implies that 
\begin{align}
|\hat{A}| \lesssim  \frac{ 1}{|k|} | \mathcal{V}[\omega_{inhom}]| + \frac{1}{ k^2} |\hat{A}_{init}|.
\end{align}
We now use the bound \eqref{s:e} to control the $\omega_{inhom}$ contribution: 
\begin{align} \n
|\hat{A}| \lesssim &  \frac{ 1}{|k| \Re(\lambda)} |k|^{1/3} \| (1+y)^3 \omega_{init} \|_{L^2_y} + \frac{1}{ k^2} |\hat{A} _{init}| \\ \n
\lesssim &   \frac{1}{|k|^{4/3}} \| (1+y)^3 \omega_{init} \|_{L^2_y} + \frac{1}{k^2} |\hat{A}_{init}|
\end{align}
\vspace{2 mm}

\noindent \textit{Case 2: $|\lambda + i k |k| | \le \frac12 |k|^2$} This case is more delicate and relies upon a cancellation of the cross term between $A$ and the leading order hydrostatic quantity $\omega^{(0)}_H$. To identify this cancellation, we introduce $f_H := A \omega_H^{(0)}$. We write \eqref{got1} together with the equation on $f_H$, which reads
\begin{subequations}
\begin{align} \label{first:eq}
(\lambda + i k + ik|k|) \hat{A} = &  ik  \mathcal{V}[ f_H]  +  i k\lambda_{\ast} \hat{A} \mathcal{V}[ \Omega_{BL}  ] + ik \hat{A} \mathcal{V}[ \omega_H^{(tail)}] + ik \mathcal{V}[\omega_{inhom}] + \hat{A}_{init}, \\ \label{sec:eq}
(\lambda + i k V_s) f_H     - \pa^2_y f_H  = & i k U''_s(y) y \hat{A} + ik U''_s \mathcal{V}_y[f_H]  
\end{align}
\end{subequations}
We take the (complex) scalar product of equation \eqref{first:eq} by $\hat{A}$ and  of \eqref{sec:eq} by $\frac{1}{- U_{s,k}''} f_H$, where $U''_{s,k}$ is defined in \eqref{Usk} (the use of this weight is explained in Section \ref{sec:OmegaH}).  We  integrate \eqref{sec:eq} by parts in $y$, and subsequently take the real part. This produces the identity 
\begin{align} \n
&\Re(\lambda) |\hat{A}|^2 + \Re(\lambda) \| \frac{1}{(- U_{s,k}'')^{1/2}} f_H \|_{L^2}^2 +  \| \frac{1}{(- U_{s,k}'')^{1/2}} \p_y f_H \|_{L^2}^2 \\ \n
= & \Re \Big( i k \mathcal{V}[f_H] \overline{\hat{A}} - \langle ik y \hat{A}, f_H \rangle \Big) + \Re \langle i k \frac{U''_{s,k} - U''_s}{U''_{s,k}} y \hat{A} , f_H \rangle   + \Re  \langle \frac{(U_{s,k}'')'}{|U_{s,k}''|^2} \p_y f_H, f_H \rangle - \Re \langle i k \mathcal{V}_y[f_H], f_H \rangle \\ \label{cel1}
&+ \Re \langle i k \frac{U''_{s,k} - U''_s}{U''_{s,k}} \mathcal{V}_y[f_H], f_H \rangle + \Re ( i k \lambda_{\ast} \mathcal{V}[\Omega_{BL}] |\hat{A}|^2)  + \Re ( i k \mathcal{V}[\omega_H^{(tail)}] |\hat{A}|^2 ) + \Re ( ik \mathcal{V}[\omega_{inhom}] \overline{\hat{A}} ) + \Re ( \hat{A}_{init} \overline{\hat{A}} ). 
\end{align}
We will now extract a cancellation from the first two terms on the right-hand side of \eqref{cel1}. An integration by parts gives 
\begin{align} \n
\Re \Big( i k \mathcal{V}[f_H] \overline{\hat{A}} - \langle ik y \hat{A}, f_H \rangle \Big) = & \Re \Big( i k \mathcal{V}[f_H] \overline{\hat{A}} + \langle ik y \hat{A}, \p_y \mathcal{U}_y[f_H] \rangle \Big) \\ \n
= & \Re \Big( i k \mathcal{V}[f_H] \overline{\hat{A}} - \langle ik \hat{A},  \mathcal{U}_y[f_H] \rangle \Big) \\ \n
= & \Re \Big( i k \mathcal{V}[f_H] \overline{\hat{A}} +  ik \hat{A}  \mathcal{V}[\overline{f_H}]  \Big) \\ \label{cancellation:1}
= & 0, 
\end{align}
where we have used for any two complex numbers $a, b \in \mathbb{C}$, the elementary identity $\Re(i k (a \overline{b} + \overline{a} b )) = 0$. 

We then have the bound 
\begin{align*}
\Big| \langle i k \frac{U''_{s,k} - U''_s}{U''_{s,k}} y \hat{A} , f_H \rangle \Big| & \le |k| \Big| \langle (U''_s - U''_{s,k})^{1/2} y  \hat{A} , \frac{f_H}{(-U''_{s,k})^{1/2}} \rangle 
\Big|  \\
& \le  |k|^{2/3} \|\frac{1}{(1+y)}\|_{L^2} |\hat{A}| \, \| \frac{f_H}{(-U''_{s,k})^{1/2}} \|_{L^2} \\
& \lesssim   |k|^{2/3} \big( |\hat{A}|^2 +  \| \frac{f_H}{(-U''_{s,k})^{1/2}} \|_{L^2}^2   \big) 
\end{align*}
The next three terms can be treated exactly as in the proof of Lemma \ref{lemma_hydro}, Section \ref{sec:OmegaH}, taking $f = f_H$. One keypoint is the cancellation  
\begin{align} \label{cancellation:hydro}
\Re \langle i k \mathcal{V}_y[f_H], f_H \rangle  = 0. 
\end{align}
We find (see Lemma \ref{lemma_hydro} for all necessary details):  
\begin{align*}
\Big| \langle   \frac{(U_{s,k}'')'}{(U_{s,k}'')^2}  \p_y f_H, f_H \rangle \Big| -  \frac{1}{2} \| \frac{1}{(-U_{s,k}'')^{1/2}} \p_y f_H \|_{L^2}^2 \lesssim   \| \frac{1}{(-U_{s,k}'')^{1/2}} f_H \|_{L^2}^2 
\end{align*}
\begin{align*}
\Big| \langle \frac{i k (U''_s - U''_{s,k})}{U''_{s,k}} \mathcal{V}_y [f_H], f_H \rangle \Big| &   \lesssim |k|^{2/3} \| \frac{1}{(-U_{s,k}'')^{1/2}} f_H \|_{L^2}^2. 
\end{align*}
Inserting into \eqref{cel1}, we get 
\begin{align} \n
&\Re(\lambda) |\hat{A}|^2 + \Re(\lambda) \| \frac{1}{(- U_{s,k}'')^{1/2}} f_H \|_{L^2}^2 +  \frac{1}{2}\| \frac{1}{(- U_{s,k}'')^{1/2}} \p_y f_H \|_{L^2}^2 \\ \n
\lesssim & |k|^{2/3} |\hat{A}|^2  +  (|k|^{2/3} + 1) \| \frac{1}{(-U_{s,k}'')^{1/2}} f_H \|_{L^2}^2 +      \Re ( i k \lambda_{\ast} \mathcal{V}[\Omega_{BL}] |\hat{A}|^2)  + \Re ( i k \mathcal{V}[\omega_H^{(tail)}] |\hat{A}|^2 ) \\  \n
&+ \Re ( ik \mathcal{V}[\omega_{inhom}] \overline{\hat{A}} ) + \Re ( \hat{A}_{init} \overline{\hat{A}} ). 
\end{align}
By conditions \eqref{k:hyp}-\eqref{lambda:hyp}, the first two terms at the right-hand side can be absorbed for $K_\ast$ large enough, resulting in 
\begin{align} \n
&\Re(\lambda) |\hat{A}|^2 + \Re(\lambda) \| \frac{1}{(- U_{s,k}'')^{1/2}} f_H \|_{L^2}^2 +  \| \frac{1}{(- U_{s,k}'')^{1/2}} \p_y f_H \|_{L^2}^2 \\  \label{cel2}
\lesssim &       \Re ( i k \lambda_{\ast} \mathcal{V}[\Omega_{BL}] |\hat{A}|^2)  + \Re ( i k \mathcal{V}[\omega_H^{(tail)}] |\hat{A}|^2 ) + \Re ( ik \mathcal{V}[\omega_{inhom}] \overline{\hat{A}} ) + \Re ( A_{init} \overline{\hat{A}} ). 
\end{align}
We now bound  the right-hand side of \eqref{cel2}. First, we have 
\begin{align} \label{blcase2}
| \Re ( i k \lambda_{\ast} \mathcal{V}[\Omega_{BL}] |\hat{A}|^2)| \lesssim &|k \lambda_{\ast}| \, \big|\mathcal{V}[\Omega_{BL}]\big| |\hat{A}|^2 
\lesssim \frac{|k|}{|\lambda|^{1/2}} \Big( 1 +  \frac{|k|}{\Re(\lambda)} \Big)  |\hat{A}|^2,
\end{align}
where we have invoked our bounds \eqref{s:a} and \eqref{s:b}. 

Second, we have by \eqref{s:d}: 
\begin{align}
| \Re ( i k \mathcal{V}[\omega_H^{(tail)}] |\hat{A}|^2 )| \lesssim  \frac{k^2}{ \Re(\lambda) |\lambda|^{\frac12} }   \big( 1 +  \frac{|k|}{\Re(\lambda)} \big) |\hat{A}|^2
\end{align}
where we have invoked \eqref{s:d}

Injecting back into \eqref{cel2}, we obtain 
\begin{align} \n
&\Re(\lambda) |\hat{A}|^2 + \Re(\lambda) \| \frac{1}{(- U_{s,k}'')^{1/2}} f_H \|_{L^2}^2 +  \| \frac{1}{(- U_{s,k}'')^{1/2}} \p_y f_H \|_{L^2}^2 \\  \label{cel3}
\lesssim & \underbrace{\Big( \frac{|k|}{|\lambda|^{1/2}} +  \frac{k^2}{\Re(\lambda)|\lambda|^{\frac12}} + \frac{|k|^3}{ \Re(\lambda)^2 |\lambda|^{\frac12}}\Big)  }_{\bold{n}(\lambda, k)}  |\hat{A}|^2 + |\Re ( ik \mathcal{V}[\omega_{inhom}] \overline{\hat{A}} )| +| \Re ( \hat{A}_{init} \overline{\hat{A}} )|. 
\end{align}
To bound the Fourier-Laplace multiplier, $\bold{n}(\lambda, k)$, we have to observe that 
\begin{align} \label{rest:1}
|\lambda + i k |k| | \le \frac12 |k|^2 \Rightarrow |\Im(\lambda)| \ge \frac{1}{2} |k|^2 \Rightarrow |\lambda| \ge \frac12 |k|^2. 
\end{align}
Using this observation, we find that 
\begin{align} \n
|\bold{n}(\lambda, k)| \lesssim & 1 + \frac{|k|}{\Re(\lambda)}+ \frac{k^2}{\Re(\lambda)^2}  \\ \n
\lesssim& \Big( \frac{1}{\Re(\lambda)} + \frac{1}{K_\ast^{3/2} \Re(\lambda)} +   \frac{1}{K_\ast^{3/2}}  \Big) \Re(\lambda) \ll  \Re(\lambda)
\end{align}
using \eqref{k:hyp} and \eqref{lambda:hyp}. Therefore, these terms can be absorbed to the left-hand side of \eqref{cel3}. Doing so produces the bound
\begin{align} \n
&\Re(\lambda) |\hat{A}|^2 + \Re(\lambda) \| \frac{1}{(- U_{s,k}'')^{1/2}} f_H \|_{L^2}^2 +  \| \frac{1}{(- U_{s,k}'')^{1/2}} \p_y f_H \|_{L^2}^2  \\  \label{cel4}
\lesssim & |\Re ( ik \mathcal{V}[\omega_{inhom}] \overline{\hat{A}} )| +| \Re ( \hat{A}_{init} \overline{\hat{A}} )|. 
\end{align}
A standard Young's inequality for products gives for a $\delta > 0$, 
\begin{align*}
| ik \mathcal{V}[\omega_{inhom}] \overline{\hat{A}}| \le &   \delta \Re(\lambda) |\hat{A}|^2 +  \frac{C_{\delta}}{\Re(\lambda)} k^2 |\mathcal{V}[\omega_{inhom}]|^2 \\
\le & \delta \Re(\lambda) |\hat{A}|^2 +  \frac{C_{\delta}}{\Re(\lambda)^3} |k|^{8/3} \| (1+y)^3 \hat{\omega}_{init} \|_{L^2_y}^2 \\
\le & \delta \Re(\lambda) |\hat{A}|^2 +  C |k|^{2/3} \| (1+y)^3 \hat{\omega}_{init} \|_{L^2_y}^2
\end{align*}
while 
\begin{align*}
|  \hat{A}_{init} \overline{\hat{A}} | & \le \delta \Re(\lambda) |\hat{A}|^2 + \frac{C_{\delta}}{\Re(\lambda)} |\hat{A}_{init}|^2 \\
& \le \delta \Re(\lambda) |\hat{A}|^2 + C |k|^{-2/3} |\hat{A}_{init}|^2 
\end{align*}
Hence, 
\begin{align*}
 |\hat{A}|^2  \lesssim   \| (1+y)^3 \omega_{init} \|_{L^2_y}^2 +  |k|^{-4/3} |\hat{A}_{init}|^2 
\end{align*}
This concludes the proof of the proposition.
\end{proof}
We can now conclude the proof of our main Proposition \ref{pro:main}. 

\begin{proof}[Proof of Proposition \ref{pro:main}] Under the assumptions of Proposition  \ref{pro:main}, that are exactly \eqref{k:hyp}-\eqref{lambda:hyp}, estimate \eqref{estimA} holds:  
\begin{align*}
 |\hat{A}|  \lesssim   \| (1+y)^3 \omega_{init} \|_{L^2_y} +  |k|^{-2/3} |\hat{A}_{init}|
\end{align*}
We now come back to the decomposition of $\hat{\omega}$: 
\begin{align}
\hat{\omega} = \hat{A} \Big( \lambda_* \Omega_{BL} +   \omega_H^{(0)} +  \omega_H^{(tail)} \Big) + \omega_{inhom}.
\end{align}
By the analysis performed in Section \ref{sec:OmegaBL}, notably formula \eqref{formula_OmegaBL}, \eqref{def:xi:0} and estimate \eqref{unw:1}, we have 
\begin{equation} \label{estimOmegaBL}
 \| (1+y)^3 \Omega_{BL} \|_{L^2} \lesssim  \|(1+y)^3 \xi_0 \|_{L^2} +   \|(1+y)^3 \Xi_0 \|_{L^2} \lesssim |\lambda|^{1/4} + \frac{|k|}{|\lambda|^{5/4}} \lesssim |\lambda|^{1/4} + |k|^{1/6}  \lesssim |\lambda|^{1/4}
 \end{equation}
By the analysis performed in Section \ref{sec:OmegaH}, notably \eqref{bdH0}, we have 
$$ \| (1+y)^3 \omega_H^{(0)}\|_{L^2} \lesssim \|\frac{1}{(-U''_{s,k})^{1/2}}  \omega_H^{(0)}\|_{L^2}  \lesssim \frac{|k|}{\Re(\lambda)} \lesssim |k|^{1/3}  $$
Using also \eqref{omegaHtail}, \eqref{bd:l0}, and estimate \eqref{bdH1}, we get 
\begin{align*}
 \| (1+y)^3 \omega_H^{(tail)}\|_{L^2}&  \lesssim  |k| |\lambda_*| \|(1+y)^3 F_H\|_{L^2} \lesssim  |k| |\lambda_0|    \|\frac{1}{(-U''_{s,k})^{1/2}} F_H\|_{L^2}  \\
 &  \lesssim \frac{k}{\Re(\lambda) |\lambda|^{1/2}} \Big(1 + \frac{|k|}{\Re \lambda}\Big) \lesssim  |k|^{1/3}.
  \end{align*}
 Similarly, using decomposition \eqref{decompo_omegaIH_tail}, \eqref{bd:l0:t}, \eqref{estimOmegaBL}, \eqref{bdIH0}, we find
 \begin{equation} \label{estim_omega_inhom}
 \begin{aligned}
  \| (1+y)^3 \omega_{inhom}\|_{L^2}& \lesssim \frac{1}{\Re(\lambda)} \Big( |\lambda|^{1/4} + |k|^{1/6} + 1  + \frac{|k|}{\Re(\lambda)|\lambda|^{1/2}}\Big)  \|\frac{1}{(-U''_{s,k})^{1/2}} \hat{\omega}_{init}\|_{L^2} \\
  &  \lesssim \frac{|k|^{1/3}}{\Re(\lambda)} \Big( |\lambda|^{1/4} + |k|^{1/6} + 1  + \frac{|k|}{\Re(\lambda)|\lambda|^{1/2}}\Big)  \|(1+y)^3 \hat{\omega}_{init}\|_{L^2} \\
  & \lesssim   |k|^{-1/3}  |\lambda|^{1/4}  \|(1+y)^3 \hat{\omega}_{init}\|_{L^2}.
  \end{aligned} 
  \end{equation}
 \end{proof}
 Together with the estimate for $\lambda_\ast$, {\it cf. } \eqref{s:a} and the estimate \eqref{estimA} for $\hat{A}$, we end up with 
 \begin{align}
 \|(\hat{\omega}, \hat{A}) \|_H  \lesssim &  |k|^{1/3} |\lambda|^{1/4}  \Big(  \| (1+y)^3 \hat{\omega}_{init} \|_{L^2_y} +  |k|^{-2/3} |\hat{A}_{init}| \Big) 
 \end{align}
 which yields the estimate of the proposition.

\section{Construction of $\Omega_{BL}$} \label{sec:OmegaBL}

This section is devoted to the construction of $\Omega_{BL}$,  solution to \eqref{gtgt2}, under the assumptions \eqref{k:hyp}-\eqref{lambda:hyp}. We will achieve this $\Omega_{BL}$ as a sum:
\begin{align} \label{blexp}
\Omega_{BL} := \sum_{j = 0}^{\infty} (\xi^{(j)} + \Xi^{(j)}),
\end{align}
where we initialize the iteration by defining: 
\begin{align} \label{def:xi:0}
\xi^{(0)}(\lambda, y) := \lambda^{\frac12} e^{- \lambda^{\frac12}y}, \qquad \alpha_j := \mathcal{U}[\Xi^{(j)}]. 
\end{align}
where $\lambda^{1/2}$ is the square root of $\lambda$ with positive real part. Notice that $\mathcal{U}[\xi^{(0)}] = 1$. We now define, for $j \ge 0$, the profiles $\Xi^{(j)} = \Xi^{(j)}(\lambda, k, y)$ through the following equation:
\begin{subequations}
\begin{align} \label{eq:Xij}
& (\lambda + ik V_s)\Xi^{(j)}  - \pa_y^2 \Xi^{(j)} =  - i k V_s \xi^{(j)}, \\
& \Xi^{(j)}|_{y = 0} = 0.
\end{align}
\end{subequations}
(see below for well-posedness).  We then define, for $j \ge 1$, the following ``heat" profiles:
\begin{subequations}
\begin{align} \label{m:11}
&\lambda \xi^{(j)} - \p_y^2 \xi^{(j)} = 0, \\ \label{m:12}
&\mathcal{U}[\xi^{(j)}] = - \mathcal{U}[\Xi^{(j-1)}].
\end{align}
\end{subequations}
This equation admits explicit solutions 
\begin{align} \label{def:xi:j}
\xi^{(j)} =  - \mathcal{U}[\Xi^{(j-1)}] \lambda^{\frac12} e^{- \lambda^{\frac12}y} = - \mathcal{U}[\Xi^{(j-1)}] \xi^{(0)} = - \alpha_{j-1} \xi^{(0)}, \qquad j \ge 1
\end{align}
Inserting this into \eqref{eq:Xij}, we obtain that 
\begin{align} \label{dueto}
\Xi^{(j)} = - \alpha_{j-1} \Xi^{(0)}, \qquad j \ge 1,
\end{align}
where the profile $\Xi^{(0)}$ satisfies 
\begin{subequations}
\begin{align} \label{pi1}
& (\lambda + ik V_s)\Xi^{(0)}  - \pa_y^2 \Xi^{(0)} =   - i k V_s \xi^{(0)}, \\
\label{pi2}
& \Xi^{(0)}|_{y = 0} = 0.
\end{align}
\end{subequations}
Inserting \eqref{dueto} into the definition of $\alpha_j$, we obtain the relation 
\begin{align} \label{alphajind}
\forall j \ge 1, \quad \alpha_{j} = - \alpha_0 \alpha_{j-1}, \quad \alpha_0 =  \mathcal{U}[\Xi^{(0)}].
\end{align}
From all these relations, we see that once the well-posedness of \eqref{pi1}-\eqref{pi2} will be shown, all terms $(\xi^{(j)}, \Xi^{(j)})$  in the expansion \eqref{blexp} will be well-defined through formulae \eqref{def:xi:0}, \eqref{def:xi:j} and \eqref{dueto}, having noticed that $\alpha_j = \alpha_0 (- \alpha_0)^j$. Moreover, the convergence of the sum in \eqref{blexp}  will hold if $|\alpha_0| < 1$, which will be shown to be true under  \eqref{k:hyp}-\eqref{lambda:hyp}. 

The well-posedness of  \eqref{pi1}-\eqref{pi2} is settled in
\begin{lem} \label{lem_xi0}
System  \eqref{pi1}-\eqref{pi2}  has a unique solution  $\Xi^{(0)}$  satisfying for all $m \ge 0$:   
\begin{align} \label{unw:1}
 \| y^m \Xi^{(0)}  \|_{L^2}^2 \lesssim_m  \frac{k^2}{ |\lambda|^{m + 5/2}}, \quad  \| y^m \p_y \Xi^{(0)}  \|_{L^2}^2 \lesssim_m  \frac{k^2}{ |\lambda|^{m + 3/2}}.
\end{align}
where the implicit constant in the above inequalities depends on $m$. 
\end{lem}
\begin{proof}
We just detail the {\it a priori} estimates, the construction of the solution being then classical.  We will make  use of the fact that $\Re(\lambda^{1/2}) \approx |\lambda|^{1/2}$. More precisely, if $\Re(\lambda) > 0$, then
$$ \Re(\lambda^{1/2}) \le |\lambda|^{1/2} \le \sqrt{3}  \Re(\lambda^{1/2}). $$
Indeed, the first inequality is trivial. For the second one, we write $\lambda^{1/2} = a + i b$, $a > 0$, so that $\lambda = a^2 - b^2 + 2 i a b$. Condition $\Re(\lambda) > 0$ implies $a \ge |b|$, so that  
$$ |\lambda| \le a^2 - b^2 + 2 a \, |b| \le 3 a^2 = 3 \Re(\lambda^{1/2})^2. $$

 We now take the (complex) scalar product of  \eqref{pi1} with $ y^{2m} \Xi^{(0)} $, and take the real part. This produces
\begin{align} \label{ineqXi0}
 \Re(\lambda) \| y^m \Xi^{(0)} \|_{L^2}^2  +  \| y^m  \p_y \Xi^{(0)} \|_{L^2}^2 - m (2m-1) \| y^{m-1} \Xi^{(0)} \|_{L^2_z}^2 =  - \Re  \langle  i k V_s \xi^{(0)}, \Xi^{(0)} y^{2m} \rangle.
\end{align}
We estimate the  right-hand side, using  $|V_s(y)| \le  \|V'_s\|_{\infty} \, y$ :  
\begin{align} \n
| \langle  i k V_s \xi^{(0)}, \Xi^{(0)} y^{2m}\rangle | & \le \frac{|\lambda|^{1/2} \|V'_s\|_{\infty} |k|}{\Re(\lambda^{1/2})^{m+1}}   \|   z^{m+1} e^{-z}\vert_{z=\Re(\lambda^{1/2}) y} \|_{L^2}  \|y^m \Xi^{(0)}\|_{L^2}  \\ \label{lw1}
& \le \frac{ C |k|}{|\lambda|^{m/2+1/4}}   \|y^m \Xi^{(0)}\|_{L^2} \\ \label{lw2}
& \le \frac{\Re(\lambda)}{2} \| \Xi^{(0)} y^m \|_{L^2}^2  +  \frac{C^2 k^2}{2 \Re(\lambda) |\lambda|^{m+1/2}} 
\end{align}
Back to  \eqref{ineqXi0}, we deduce from the previous inequalities: 
\begin{align} \label{ineqXi0bis}
 \Re(\lambda) \| y^m \Xi^{(0)} \|_{L^2}^2  - m (2m-1) \| y^{m-1} \Xi^{(0)} \|_{L^2_z}^2 & \lesssim   \frac{k^2}{\Re(\lambda) |\lambda|^{m+1/2}} \\
  \label{ineqXi0bisbis}
    \| y^m  \p_y \Xi^{(0)} \|_{L^2}^2 & \lesssim  \frac{k}{|\lambda|^{m/2+1/4}}   \|y^m \Xi^{(0)}\|_{L^2} +  m \| y^{m-1} \Xi^{(0)} \|_{L^2_z}^2.
\end{align}
To obtain \eqref{ineqXi0bis} we simply drop the second term from the left-hand side of \eqref{ineqXi0}, apply \eqref{lw2}, and use the factor of $\frac12$ in \eqref{lw2} to absorb this contribution to the left-hand side. To obtain \eqref{ineqXi0bisbis}, we drop the first term on the left-hand side of \eqref{ineqXi0} (which is positive) which implies
\begin{align*}
 \| y^m  \p_y \Xi^{(0)} \|_{L^2}^2 \le & m (2m-1) \| y^{m-1} \Xi^{(0)} \|_{L^2_z}^2 + \Big|\Re  \langle  i k V_s \xi^{(0)}, \Xi^{(0)} y^{2m} \rangle \Big| \\
 \lesssim & m (2m-1) \| y^{m-1} \Xi^{(0)} \|_{L^2_z}^2 + \frac{ C |k|}{|\lambda|^{m/2+1/4}}   \|y^m \Xi^{(0)}\|_{L^2},
\end{align*}
where we have invoked the inequality \eqref{lw1}.

There are two cases to consider: 
\begin{itemize}
\item If $\Re(\lambda) \ge |\Im(\lambda)|$, we have $|\lambda| \approx \Re(\lambda)$, so that \eqref{ineqXi0bis} implies
\begin{align*}
& |\lambda| \| y^m \Xi^{(0)} \|_{L^2}^2   - m (2m-1) \| y^{m-1} \Xi^{(0)} \|_{L^2_z}^2 \lesssim   \frac{k^2}{ |\lambda|^{m+3/2}} 
\end{align*}
A simple induction on $m$ yields the first inequality in \eqref{unw:1},  the second one follows then from \eqref{ineqXi0bisbis}. 
\item If $|\Im(\lambda)| \ge \Re(\lambda)$, we  go back to  \eqref{pi1}, take the scalar product with $ y^{2m} \Xi^{(0)} $, but this time take the imaginary part. We find 
\begin{align} \label{ineqXi0ter}
 \Im(\lambda) \| y^m \Xi^{(0)} \|_{L^2}^2   = - k \langle  V_s \Xi^{(0)} , \Xi^{(0)} y^{2m} \rangle  - 2m \Im \langle \pa_y  \Xi^{(0)} , y^{2m-1} \Xi^0 \rangle   - \Im  \langle  i k V_s \xi^{(0)}, \Xi^{(0)} y^{2m} \rangle.
\end{align}
Proceeding as above, we have for some $C > 0$: 
$$ |   \langle  i k V_s \xi^{(0)}, \Xi^{(0)} y^{2m} \rangle | \le  \frac{ C |k|}{|\lambda|^{m/2+1/4}}   \|y^m \Xi^{(0)}\|_{L^2} \le \frac{|\Im(\lambda)|}{8} \| \Xi^{(0)} y^m \|_{L^2}^2  +  \frac{2 C^2 k^2}{ |\Im(\lambda)| |\lambda|^{m+1/2}} $$
We also have 
\begin{align*}
 | \langle \pa_y  \Xi^{(0)} , y^{2m-1} \Xi^0 \rangle | & \le \| y^m \pa_y  \Xi^{(0)}\|_{L^2} \|y^{m-1}  \Xi^0 \|_{L^2}  \le \frac{1}{2} \|y^{m-1}  \Xi^0 \|_{L^2}^2 +  \frac{1}{2} \| y^m \pa_y  \Xi^{(0)}\|_{L^2}^2 \\
 & \le C \|y^{m-1}  \Xi^0 \|_{L^2}^2 +  \frac{C |k|}{|\lambda|^{m/2+1/4}}   \|y^m \Xi^{(0)}\|_{L^2} \\ 
 & \le  \frac{|\Im(\lambda)|}{8} \| \Xi^{(0)} y^m \|_{L^2}^2  + \frac{C' k^2}{|\lambda|^{m+1/2} |\Im(\lambda)|}. 
\end{align*}
Note that we have used   \eqref{ineqXi0bisbis} to go from the second to the third inequality.  Moreover, we have 
\begin{align*}
|k \langle  V_s \Xi^{(0)} , \Xi^{(0)} y^{2m} \rangle| & \le |k| \, \|V'_s\|_{\infty} \|y^{m+1/2}  \Xi^{(0)}\|_{L^2}^2 \le  |k| \, \|V'_s\|_{\infty} \|y^{m}  \Xi^{(0)}\|_{L^2} \|y^{m+1}  \Xi^{(0)}\|_{L^2} \\
& \le \frac{|\Im(\lambda)|}{8} \| \Xi^{(0)} y^m \|_{L^2}^2  + \frac{C k^2}{|\Im(\lambda)|}  \|y^{m+1}  \Xi^{(0)}\|_{L^2}^2 \\
& \le  \frac{|\Im(\lambda)|}{8} \| \Xi^{(0)} y^m \|_{L^2}^2  +  \frac{C' k^2}{|\Im(\lambda)| \, \Re(\lambda)}  \|y^{m}  \Xi^{(0)}\|_{L^2}^2 + \frac{C' k^4}{|\Im(\lambda)| \, \Re(\lambda)^2 |\lambda|^{m+3/2}} 
\end{align*}
where the last inequality follows from \eqref{ineqXi0bis}, applied with index $m+1$ instead of $m$. 
For $K_\ast$ large enough, assumptions \eqref{k:hyp}-\eqref{lambda:hyp}, together with inequality $|\Im(\lambda)| \ge \Re(\lambda)$, yield 
$$ \frac{C' k^2}{|\Im(\lambda)| \, \Re(\lambda)}  \le  \frac{|\Im(\lambda)|}{8}, \quad   \frac{C' k^2}{|\Im(\lambda)| \, \Re(\lambda)^2} \le 1,   $$
so that we get
\begin{align*} 
|\Im(\lambda)| \| y^m \Xi^{(0)} \|_{L^2}^2   \lesssim  m  \| y^{m-1} \Xi^{(0)} \|_{L^2}^2  +  \frac{k^2}{ |\Im(\lambda)| |\lambda|^{m+1/2}} +  \frac{k^2}{|\lambda|^{m+3/2}} 
\end{align*}
As $|\lambda| \approx |\Im(\lambda)|$, we find 
\begin{align*} 
|\lambda| \| y^m \Xi^{(0)} \|_{L^2}^2   \lesssim  m  \| y^{m-1} \Xi^{(0)} \|_{L^2}^2  +    \frac{k^2}{|\lambda|^{m+3/2}}   
\end{align*}
A simple induction on $m$ yields the first inequality in \eqref{unw:1}. The second one follows then from \eqref{ineqXi0bisbis}. This concludes the proof. 
\end{itemize}

%
\end{proof}

A corollary to our construction is the following: 
\begin{cor}  \label{corXi0}
Under assumptions \eqref{k:hyp}-\eqref{lambda:hyp},  the constant $\alpha_0 = \mathcal{U}[\Xi^{(0)}]$ satisfies  
\begin{align} \label{alpha0bd}
|\alpha_0| <& 1 
\end{align}
As a consequence, the function $\Omega_{BL}$ introduced in \eqref{blexp} is well-defined, belongs to $L^2(y^m dy)$ for all $m \ge 0$, and is a solution of \eqref{gtgt2}. Moreover, it satisfies the estimate: 
\begin{align}
\label{prof:V:prop}
\sup_{y \ge 0} |\mathcal{V}_y[\Omega_{BL}]| \le  \mathcal{V}[|\Omega_{BL}|]  \lesssim & \frac{1}{|\lambda|^{\frac12}}.
\end{align}
\end{cor} 
\begin{proof} For any function $f = f(y)$ integrable over $\R_+$, any $\delta > 0$, we have 
\begin{align}
\nonumber
\int_{\R_+} f & = \int_0^\delta f + \int_\delta^{+\infty} f =  \int_0^\delta f + \int_\delta^{+\infty} \frac{1}{y} (yf) \\
\nonumber
& \le \sqrt{\delta} \|f\|_{L^2} + \Big( \int_{\delta}^\infty \frac{1}{y^2}  \Big)^{1/2} \|y f \|_{L^2} =  \sqrt{\delta} \|f\|_{L^2} + \frac{1}{\sqrt{\delta}} \|y f \|_{L^2} \\
\label{interpol}
& \le \|f\|_{L^2}^{1/2} \|y f\|_{L^2}^{1/2} 
\end{align}
where we optimized in $\delta$ to get the last bound. It follows from this interpolation inequality and from the estimates \eqref{unw:1} that
\begin{align} \label{estimalpha0}
|\alpha_0| \le \int_{\R_+} |\Xi^{(0)}| & \le \|\Xi^{(0)}\|^{1/2}_{L^2} \|y \Xi^{(0)}\|^{1/2}_{L^2}  \lesssim   \Big( \frac{k}{|\lambda|^{5/4}} \Big)^{1/2}  \Big( \frac{k}{|\lambda|^{7/4}}  \Big)^{1/2} \lesssim \frac{k}{|\lambda|^{3/2}} < 1 
\end{align}
We deduce from the analysis at the beginning of Section \ref{sec:OmegaBL} and from \eqref{alpha0bd} that the sum introduced in \eqref{blexp} converges:  
\begin{align}
 \Omega_{BL} & = \sum_{j = 0}^{\infty} (\xi^{(j)} + \Xi^{(j)}) \\
 & = \xi^{(0)} + \Xi^{(0)}  - \sum_{j \ge 1} \alpha_{j-1} ( \xi^{(0)} + \Xi^{(0)}) = \big(1-   \sum_{j \ge 1} (-\alpha_0)^{j-1} \alpha_0\big)  (\xi^{(0)} + \Xi^{(0)}) \\
 & = \big( 1 -  \frac{\alpha_0}{1 + \alpha_0} \big) (\xi^{(0)} + \Xi^{(0)})  \label{formula_OmegaBL}
\end{align}
As $\xi^{(0)}$ decays exponentially, and as $\Xi^{(0)} \in L^2(y^m dy)$ for all $m \ge 0$ by estimates \eqref{unw:1}, $\Omega_{BL} \in L^2(y^m dy)$ for all $m \ge 0$.

For the bound \eqref{prof:V:prop}, we write 
\begin{align*}
 \mathcal{V}[|\Omega_{BL}|] \lesssim \mathcal{V}[|\xi^{(0)}|] +  \mathcal{V}[|\Xi^{(0)}|] \lesssim \frac{1}{|\lambda|^{1/2}} + \mathcal{V}[|\Xi^{(0)}|] 
\end{align*}
where the first term at the right-hand side comes from an explicit computation, based on formula \eqref{def:xi:0}. For the second term, we integrate by parts to get: 
\begin{equation} \label{boundVXi0}
\begin{aligned}
\mathcal{V}[|\Xi^{(0)}|] & = \int_{0}^{+\infty} \big( \int_{y}^{+\infty} |\Xi^{(0)}| \big) dy =  \int_{0}^{+\infty} y |\Xi^{(0)}(y)| dy \le \| y  \Xi^{(0)} \|_{L^2}^{1/2} \| y^2  \Xi_0 \|_{L^2}^{1/2} \\
& \lesssim \frac{k}{|\lambda|^2} \lesssim \frac{1}{|\lambda|^{1/2}}
\end{aligned}
\end{equation}
Here we have used successively the interpolation inequality \eqref{interpol} with $f = y |\Xi^{(0)}|$  and the bounds \eqref{unw:1}. This concludes the proof. 
\end{proof}

\section{Construction of Hydrostatic Profiles} \label{sec:OmegaH}

In this section, we want to construct all of the ``hydrostatic" profiles appearing in our analysis. These include $F_H, \omega_H^{(0)},$ and $\omega_{IH}^{(0)}$. The abstract problem behind this construction is:
\begin{equation} \label{cdefIH:abs}
\begin{aligned}
& (\lambda + i k V_s) f  - ik U''_s \mathcal{V}_y[ f] - \pa^2_y f =  \mathcal{R},  \\
& \p_y f||_{y = 0}  = 0,
\end{aligned}
\end{equation}
The point is to solve this problem under conditions \eqref{k:hyp}-\eqref{lambda:hyp}. The difficulty lies in the stretching term  $ik U''_s \mathcal{V}_y[ f]$, which is {\it a priori} $O(|k|)$, and can not be absorbed in the standard energy estimate unless $\Re(\lambda)  \approx |k|$, which only provides local well-posedness for data analytic in $x$.  This difficulty is by now classical and appears in the analysis of several anisotropic systems, including hydrostatic Euler equations \cite{} or Prandtl equations \cite{MR1617542,MR2975371}. In the case of hydrostatic Euler, it is well-known that generically, analyticity is needed for well-posedness, just as in the case of the Triple-Deck system \cite{Renardy09}. But when data are concave,  Sobolev stability estimates can be derived \cite{MR1690189}. Roughly,  the idea is to test against $-\frac{1}{U''_s}f $ instead of $f$, and to use the cancellation 
\begin{align*}
\Re \langle  -ik U''_s \mathcal{V}_y[ f] , \frac{1}{U''_s} f \rangle & = \Re    \langle ik \mathcal{V}_y[ f] , f \rangle =  - \Re   \langle   ik \int_{+\infty}^y f  , \int_{+\infty}^y  f \rangle  = 0. 
\end{align*}
We will here adopt the same kind of weighted estimates. Still, there are  difficulties compared to the case of hydrostatic Euler equations. First, the diffusion term $-\pa^2_y f$ creates additional  terms, including boundary terms after integration by parts. This is  why we need the artificial homogeneous Neumann condition $\pa_y f = 0$, to be compared with the "real" inhomogeneous condition \eqref{neumann_omega} satisfied by the vorticity  of the Triple-Deck system, or equivalently with condition \eqref{eq:vort:c}. This also explains the need for the complicated iterative scheme described in paragraph \ref{subsec:iteration}, with the addition of  boundary layer terms that allows to restore the real boundary condition. We remind that this scheme has strong similarities with the one of \cite{GVMM}. 

Another difficulty comes from the fact that we want to include in our analysis shear flows such that $-U''_s$  decays very fast at infinity, in which case  the hydrostatic weight  $-1/U''_s$  would impose too much decay on the data. To overcome this issue, our idea is to  consider the weight $1/-U''_{s,k}$, which has been defined in \eqref{Usk}. Our main result on the abstract problem \eqref{cdefIH:abs} is the following:
\begin{lem} \label{lemma_hydro}
Under \eqref{k:hyp}-\eqref{lambda:hyp}, system  \eqref{cdefIH:abs} has a unique solution $f$ satisfying: 
\begin{align} \label{mainapriori}
\Re(\lambda) \| \frac{1}{(-U_{s,k}'')^{1/2}} f \|_{L^2}^2 + \|\frac{1}{(-U_{s,k}'')^{1/2}} \p_y f \|_{L^2}^2 \lesssim \frac{1}{\Re(\lambda)} \| \frac{1}{(-U_{s,k}'')^{1/2}} \mathcal{R} \|_{L^2}^2.
\end{align}
\end{lem}
\begin{proof}
We again focus on the estimate, the construction following from standard arguments. We take the (complex) scalar product of  \eqref{eq:vort:bar:a} with $f \frac{1}{-U_{s,k}''}$ and take the real part:    
\begin{align} 
  &\Re(\lambda) \| \frac{1}{(-U_{s,k}'')^{1/2}} f \|_{L^2}^2 + \|\frac{1}{(-U_{s,k}'')^{1/2}} \p_y f \|_{L^2}^2 \\
  &  + \Re \langle ik \mathcal{V}_y [f], f \rangle +  \Re   \langle \frac{i k (U''_s - U''_{s,k})}{U''_{s,k}} \mathcal{V}_y [f], f \rangle + \Re \langle  \p_y f, \frac{(U_{s,k}'')'}{(U_{s,k}'')^2} f \rangle = \Re  \langle \mathcal{R}, f \frac{1}{-U_{s,k}''} \rangle 
\end{align}
Note that we have made crucial use of the Neumann condition on $f$ to integrate by parts the diffusion term. 
As explained above, the third term at the left-hand side vanishes identically. For the fourth term, we write 
\begin{align}
\Big| \langle \frac{i k (U''_s - U''_{s,k})}{U''_{s,k}} \mathcal{V}_y [f], f \rangle \Big| &  \le |k| \Big|  \langle (U''_s - U''_{s,k})^{1/2} \mathcal{V}_y [f] , \frac{1}{(-U_{s,k}'')^{1/2}} f \rangle \Big| \\
& \le  |k|^{2/3} \| (1+y)^{-2}  \mathcal{V}_y [f] \|_{L^2} \,  \| \frac{1}{(-U_{s,k}'')^{1/2}} f \|_{L^2} \le  \\
\label{Hardy1} &  \le 2 |k|^{2/3}  \| \int_y^{\infty} f \|_{L^2}  \,  \| \frac{1}{(-U_{s,k}'')^{1/2}} f \|_{L^2} \\ 
\label{Hardy2} & \le 4 |k|^{2/3} \| y f \|_{L^2} \,  \| \frac{1}{(-U_{s,k}'')^{1/2}} f \|_{L^2} \\
&  \lesssim |k|^{2/3} \| \frac{1}{(-U_{s,k}'')^{1/2}} f \|_{L^2}^2 \label{ineq:weight}
\end{align}
Here, \eqref{Hardy1} is a consequence of the usual Hardy inequality: 
$$ \| (1+y)^{-2}  \mathcal{V}_y [f] \|_{L^2} \le  \| (1+y)^{-1}  \mathcal{V}_y [f] \|_{L^2} \le 2 \| f \|_{L^2} $$
 while \eqref{Hardy2} comes from the modified one: 
$$ \| F \|_{L^2} \le 2 \| y F' \|_{L^2} \quad (\text{if}  \: \lim_{y \rightarrow +\infty} F = 0) $$
which is valid for functions vanishing at infinity. Indeed, in such a case, through integration by parts:  
$$ \int_{0}^{+\infty} |F(y)|^2 dy =  - 2 \int_{0}^{+\infty} y F'(y) F(y) dy  $$
and the inequality follows from Cauchy-Schwarz.  Finally, inequality \eqref{ineq:weight} comes from the pointwise bound $y \lesssim \frac{1}{(-U''_{s,k})^{1/2}}$. 
Regarding the commutator with the diffusion, taking into account \eqref{assume:Us:3}, which implies  
$\big| (U_{s,k}'')'/U_{s,k}''\big| \lesssim 1$, we get
\begin{align*}
\Big| \langle  \p_y f,  \frac{(U_{s,k}'')'}{(U_{s,k}'')^2} f \rangle \Big| \le & \| \frac{(U_{s,k}'')'}{U_{s,k}''}\|_{L^\infty} \| \frac{1}{(-U_{s,k}'')^{1/2}} \p_y f \|_{L^2} \| \| \frac{1}{(-U_{s,k}'')^{1/2})} f \|_{L^2}  \\
\le & \frac{1}{2} \| \frac{1}{(-U_{s,k}'')^{1/2}} \p_y f \|_{L^2}^2 +C \| \frac{1}{(-U_{s,k}'')^{1/2}} f \|_{L^2}^2 \\ 
\end{align*}
We finally bound the source term via 
\begin{align*}
| \langle \mathcal{R}, f \frac{1}{-U_{s,k}''} \rangle| \lesssim \| \frac{1}{(-U_{s,k}'')^{1/2}} \mathcal{R} \|_{L^2_y} \| \frac{1}{(-U_{s,k}'')^{1/2}} f \|_{L^2_y} \le  \frac{1}{2 \Re(\lambda)}\| \frac{1}{(-U_{s,k}'')^{1/2}} \mathcal{R} \|_{L^2_y}^2 + \frac{1}{2} \Re(\lambda) \| \frac{1}{(-U_{s,k}'')^{1/2}} f \|_{L^2_y}^2.
\end{align*}
Gathering all these estimates, and using \eqref{k:hyp}-\eqref{lambda:hyp} to absorb the terms in $\| \frac{1}{(-U_{s,k}'')^{1/2}} f \|_{L^2_y}^2$ that are at the right-hand side (notably the one from \eqref{ineq:weight}), we obtain \eqref{mainapriori}. This concludes the proof. 
\end{proof} 
We are now ready to construct the ``hydrostatic" quantities, $F_H, \omega_H^{(0)},$ and $\omega_{IH}^{(0)}$.
\begin{cor} 
Under \eqref{k:hyp}-\eqref{lambda:hyp},  systems \eqref{gtgt}, \eqref{eq:vort:bar:a}, and \eqref{eq:vort:IH} have solutions $F_H, \omega_H^{(0)},$ and $\omega_{IH}^{(0)}$ respectively,  obeying the following estimates: 
\begin{align} \label{bdH1}
\Re(\lambda) \| \frac{1}{(- U_{s,k}'')^{1/2}} F_H \|_{L^2}^2 + \|\frac{1}{(- U_{s,k}'')^{1/2}} \p_y F_H \|_{L^2}^2 \lesssim & \frac{1}{\Re(\lambda)} \sup_{y \ge 0} |\mathcal{V}_y[\Omega_{BL}]|^2 \lesssim  \frac{1}{\Re(\lambda) |\lambda|}\\ \label{bdH0}
\Re(\lambda) \| \frac{1}{(- U_{s,k}'')^{1/2}} \omega_H^{(0)} \|_{L^2}^2 + \|\frac{1}{(- U_{s,k}'')^{1/2}} \p_y \omega_H^{(0)} \|_{L^2}^2 \lesssim & \frac{ |k|^2}{\Re(\lambda)}, \\ \label{bdIH0}
\Re(\lambda) \| \frac{1}{(- U_{s,k}'')^{1/2}} \omega_{IH}^{(0)} \|_{L^2}^2 + \|\frac{1}{(- U_{s,k}'')^{1/2}} \p_y \omega_{IH}^{(0)} \|_{L^2}^2 \lesssim & \frac{ 1 }{\Re(\lambda)} \| \frac{1}{(- U_{s,k}'')^{1/2}} \omega_{init} \|_{L^2_y}^2.
\end{align}
where $U''_{s,k}$ was defined in \eqref{Usk}. 
\end{cor}
\begin{proof} This follows by applying Lemma \ref{lemma_hydro}, upon choosing $\mathcal{R}$ to be equal to $U_s'' \mathcal{V}_y[\Omega_{BL}],ik U_s'' y$, and $\omega_{init}$ respectively. We make use of the fact that 
$$ \frac{-U_s''}{(-U''_{s,k})^{1/2}} \le (-U''_s)^{1/2} \in L^2(\R_+)$$
Also, regarding \eqref{bdH1}, we use \eqref{prof:V:prop} to obtain the second bound. 
\end{proof}

\begin{cor} \label{cor_FH}The averages satisfy the following estimate
\begin{align} 
\mathcal{U}[|F_{H}|] + \mathcal{V}[|F_H|] \lesssim \frac{1}{\Re(\lambda) |\lambda|^{1/2}},  \\ \label{bd:Lambda:0}
\mathcal{U}[|\omega_H^{(0)}|] + \mathcal{V}[|\omega_H^{(0)}|] \lesssim \frac{|k|}{\Re(\lambda)}, \\ \label{bd:ih:mean}
\mathcal{U}[|\omega_{IH}^{(0)}|] + \mathcal{V}[|\omega_{IH}^{(0)}|] \lesssim \frac{1}{\Re(\lambda)}  \| \frac{1}{(- U_{s,k}'')^{1/2}} \omega_{init} \|_{L^2_y}.  
\end{align}
\end{cor} 
\begin{proof}
From \eqref{interpol}, we have 
\begin{align}
 \mathcal{U}[|f|] & \le \int_{\R_+} |f| \le \|f\|_{L^2}^{1/2} \|yf\|_{L^2}^{1/2} \\
\mathcal{V}[|f|] \big| & \le \int_{\R_+} \int_{y}^{+\infty} |f| =  \int_{\R_+} y  |f| \le \|y f \|_{L^2}^{1/2} \|y^2 f\|_{L^2}^{1/2} 
\end{align}
This implies in particular 
\begin{align} \label{prof:F:prop}
\mathcal{U}[|F_{H}|] + \mathcal{V}[|F_H|] \lesssim \|(1+y)^3 F_H\|_{L^2} \lesssim  \| \frac{1}{(- U_{s,k}'')^{1/2}}  F_H\|_{L^2}  \\ 
\mathcal{U}[|\omega_H^{(0)}|] + \mathcal{V}[|\omega_H^{(0)}|]  \lesssim \|(1+y)^3 \omega_H^{(0)}\|_{L^2} \lesssim  \| \frac{1}{(- U_{s,k}'')^{1/2}}  \omega_H^{(0)}\|_{L^2} \\
\mathcal{U}[|\omega_{IH}^{(0)}|] + \mathcal{V}[|\omega_{IH}^{(0)}|]  \lesssim \|(1+y)^3 \omega_{IH}^{(0)}\|_{L^2} \lesssim  \| \frac{1}{(- U_{s,k}'')^{1/2}}  \omega_{IH}^{(0)}\|_{L^2}. 
\end{align}
The estimates follow then from the previous corollary. 
\end{proof}

\section{Proof of Proposition \ref{propn:s}} \label{sec:convergence}
The goal of this section is to prove Proposition  \ref{propn:s}, its main aspect being to construct, for a given $\hat{A}$,  a solution $\hat{\omega}\big[\hat{A}\big]$ to  \eqref{eq:vort:b}-\eqref{eq:vort:c}. As explained in Paragraph \ref{subsec:iteration}, we look for a solution in the form \eqref{omega:structure:1}, that involves the solutions $\bar{\omega}$ of \eqref{eq:vort:bar} and $\omega_{inhom}$ of \eqref{eq:vort:inh}. 

One has (so far formally) 
\begin{align*}
\overline{\omega} & = \omega_{H}^{(0)} + \omega_{BL}^{(0)} + \omega_{H}^{(tail)} + \omega_{BL}^{(tail)} \\
& = \omega_{H}^{(0)} + \omega_{BL}^{(0)} + \sum_{j \ge 1} \big( \omega_{H}^{(j)} + \omega_{BL}^{(j)} \big) \\
& =  \omega_{H}^{(0)} + \lambda_0 \Omega_{BL}  +  \sum_{j = 1}^{\infty}  \big( i k \lambda_{j-1} F_H + \lambda_j \Omega_{BL} \big)
\end{align*}
where, see \eqref{hg5}: 
$$ \lambda_0 := 1 - \mathcal{U}[ \omega_{H}^{(0)}], \quad \lambda_j = (-i k \mathcal{U}[F_H])^{j} \lambda_0 $$
Similarly 
\begin{align} \n
\omega_{inhom} = &\omega_{IH}^{(0)} + \omega_{IB}^{(0)} + \omega_{IH}^{(tail)}  + \omega_{IB}^{(tail)} \\ \n
= & \omega_{IH}^{(0)} + \omega_{IB}^{(0)} + \sum_{j \ge 1} \omega_{IH}^{(j)} + \sum_{j = 1}^{\infty} \omega_{IB}^{(j)} \\ 
= &  \omega_{IH}^{(0)}  + \tilde{\lambda}_0 \Omega_{BL} +    \sum_{j \ge 1}   i k \tilde{\lambda}_{j-1} F_H + \tilde{\lambda_j} \Omega_{BL} 
\end{align}
where, see \eqref{tildelambdageometric}:
$$ \tilde{\lambda}_0 =  - \mathcal{U}[\omega_{IH}^{(0)} ], \quad \tilde{\lambda}_j = (-i k \mathcal{U}[F_H])^{j} \tilde{\lambda}_0 $$

  The analysis of Sections \ref{sec:OmegaBL} and \ref{sec:OmegaH} has allowed to construct $\Omega_{BL}$, $F_H, \omega_H^{(0)}$, and $\omega_{IH}^{(0)}$. Moreover, from Corollary \ref{cor_FH}, we have 
\begin{align} 
|k \mathcal{U}[F_H]| & \lesssim \frac{k}{\Re(\lambda) |\lambda|^{1/2}}  \\ 
\label{bd:l0}
|\lambda_0| & \lesssim 1 +  \frac{|k|}{\Re(\lambda)}, \\ \label{bd:l0:t}
|\widetilde{\lambda}_0| & \lesssim \frac{1}{\Re(\lambda)}  \| \frac{1}{(- U_{s,k}'')^{1/2}} \omega_{init} \|_{L^2_y}
\end{align}
In particular, under assumptions \eqref{k:hyp}-\eqref{lambda:hyp}, one has $|k \mathcal{U}[F_H]| \le \frac{1}{2}$, which shows the convergence of the series: 
\begin{align}  \label{omegaHtail}
 \overline{\omega} & = \lambda_\ast \Omega_{BL} +  \omega_{H}^{(0)} +   \omega_{H}^{(tail)}, \quad  \text{with } \:   \lambda_\ast  := \frac{\lambda_0}{1+  i k \mathcal{U}[F_H]},\quad  \omega_{H}^{(tail)} := i k \lambda_\ast F_H, \\  
 \label{decompo_omegaIH_tail}
\omega_{inhom} & = \tilde{\lambda}_\ast \Omega_{BL} +  \omega_{IH}^{(0)} +   \omega_{IH}^{(tail)}, \quad  \text{with } \:   \tilde{\lambda}_\ast  := \frac{\tilde{\lambda}_0}{1+  i k \mathcal{U}[F_H]},\quad  \omega_{IH}^{(tail)} := i k \tilde{\lambda}_\ast  F_H.
\end{align} 
 It implies decomposition \eqref{structure:cond:om}, and  the estimates \eqref{s:a} to \eqref{s:d} follow directly from \eqref{prof:V:prop}, from the estimates of Corollary \ref{cor_FH} and from \eqref{bd:l0}-\eqref{bd:l0:t}. To prove \eqref{s:e}, we have:
 \begin{align*}
 \mathcal{V}[|\omega_{inhom}|] \lesssim &\Big(  \frac{1}{\Re(\lambda) |\lambda|^{1/2}} + \frac{1}{\Re(\lambda)} + \frac{|k|}{\Re(\lambda)^2 |\lambda|^{1/2}} \Big) \| \frac{1}{(- U_{s,k}'')^{1/2}} \hat{\omega}_{init} \|_{L^2_y} \\
 \lesssim &  \frac{1}{\Re(\lambda)} |k|^{1/3} \| (1+y)^3 \hat{\omega}_{init} \|_{L^2_y}
 \end{align*}
 where we have invoked both 
 \eqref{k:hyp}-\eqref{lambda:hyp} and the bound 
 $$ \|\frac{1}{(-U''_{s,k})^{1/2}} f \|_{L^2} \le |k|^{1/3} \|(1+y)^3 f\|_{L^2} $$
 which follows from \eqref{lowerupper1}. 
 
 This concludes the proof of the proposition.

\section{Proof of Theorem \ref{thm:main}} \label{sec:proof:thm}
Thanks to Proposition \ref{pro:main}, we  can now prove Theorem \ref{thm:main}. For technical reasons that will be made clearer below, we need a slightly modified version of Proposition \ref{pro:main}, where we replace the resolvent system \eqref{eq:vort:a}-\eqref{eq:vort:b}-\eqref{eq:vort:c} by the system 
\begin{subequations} 
\begin{align} \label{eq:vort:aN}
&(\lambda + i k + i k |k|) \hat{A}_N  = i k \mathcal{V}[ \hat{\omega}_N] + \hat{A}_{init}, \\ \label{eq:vort:bN}
& (\lambda + i k V_{s,N})  \hat{\omega}_N - ik U''_s \mathcal{V}_y[ \hat{\omega}_N] - \pa^2_y \hat{\omega}_N = i k U''_s(y) y  \hat{A}_{N} + \hat{\omega}_{init}, \\ \label{eq:vort:cN}
& \mathcal{U}[\hat{\omega}_N]  = \hat{A}_{N}, 
\end{align}
\end{subequations}
substituting to the unbounded shear flow $V_s(y) = y + U_s(y)$ the sequence of bounded shear flows   
\begin{equation} \label{approximate_shear_flow} 
V_{s,N}(y) = N \chi\big(\frac{y}{N}\big) + U_s(y), \quad N \ge 1, 
\end{equation}
for $\chi = \chi(\xi) \le \xi$ a smooth compactly supported function in $\R_+$, satisfying $\chi(\xi) = \xi$ in  $[0, \frac{1}{4}]$. It will be useful in what follows to integrate \eqref{eq:vort:aN} -- \eqref{eq:vort:cN} to the corresponding velocity formulation. 
\begin{lem} Let $(\hat{\omega}_N, \hat{A}_N)$ satisfy \eqref{eq:vort:aN} -- \eqref{eq:vort:cN}. Let $\hat{u}_N := \int_0^y \hat{\omega}_N$, $\hat{v}_N = - i k \int_0^y \hat{u}_N$, 
$\hat{u}^{init} := \int_0^y \hat{\omega}^{init}$. Then the following system is satisfied: 
\begin{subequations}
\begin{align} \label{uNeq:1}
&\lambda \hat{u}_N + i k V_s  \hat{u}_N + \hat{v}_N V_{s}' - \p_y^2 \hat{u}_N  = - i k |k| \hat{A}_N + i k \mathcal{U}_y[(V_{s,N} - V_s)  \hat{\omega}_{N}] + \hat{u}^{init}, \\ \label{uNeq:2}
&i k  \hat{u}_N + \p_y \hat{v}_N = 0, \\ \label{uNeq:3}
&[\hat{u}_N, \hat{v}_N]|_{y = 0} = 0, \qquad \hat{u}_N|_{y = \infty} = \hat{A}_N.  
\end{align}
\end{subequations}
\end{lem}
\begin{proof} This follows essentially verbatim to the proof of Lemma \ref{lem:u:int}. 
\end{proof}

All estimates used to show Proposition \ref{pro:main} apply to the system \eqref{eq:vort:aN} -- \eqref{eq:vort:cN} so that we can state: 
 \begin{pro} \label{pro:main2}
 There exists absolute positive constants $K_{\ast}$, $C_0$, $k_0$ and $M$, such that for all $N\ge 1$, $|k| \ge k_0$,  all $\lambda$ with $\Re(\lambda) \ge K_{\ast} |k|^{2/3}$, and all data $(\hat{\omega}_{init}, A_{init}) \in H$, {\it cf.} definition  \eqref{defi:H},   system \eqref{eq:vort:aN}-\eqref{eq:vort:bN}-\eqref{eq:vort:cN} has a unique solution satisfying 
$$  \| (\hat{\omega}_N , \hat{A}_N)\|_H  \le C_0  |k|^{1/3} |\lambda|^{1/4}  \| (\hat{\omega}_{init} , \hat{A}_{init})\|_H. $$
 \end{pro}
We insist that the control at the right-hand side is uniform in $N$, notably because $\|V'_{s,N}\|_{L^\infty}$ is bounded uniformly in $N$. We then state refined resolvent estimates: \begin{lem} \label{lem:resolvent_estim}
The solution $(\hat{\omega}_N, \hat{A}_N)$ of the resolvent system \eqref{eq:vort:aN}-\eqref{eq:vort:bN}-\eqref{eq:vort:cN} given by Proposition \ref{pro:main2} satisfies: for all $|k| \ge k_0$, $N \ge c |k|^{4/3}$, where $c$ is a large universal constant, and for all $\lambda$ such that $\Re(\lambda) \ge K_* |k|^{2/3}$,  
\begin{align} \label{jimmybutler:1}
\| (\hat{\omega}_N, \hat{A}_N)\|_H  \le C_N \frac{|k|^{s_0}}{|\lambda|} \| (\hat{\omega}_{init} , \hat{A}_{init})\|_H
\end{align}
as well as 
\begin{align} \label{jimmybutler:2}
\| ((1+y)^{-1}\hat{\omega}_N, \hat{A}_N)\|_H  \le C_0 \frac{|k|^{s_0}}{|\lambda|} \| (\hat{\omega}_{init} , \hat{A}_{init})\|_H
\end{align}
where $C_0$, $s_0$ are absolute constants, while $C_N$ possibly depends on $N$. 
\end{lem}
\begin{proof}The proof proceeds by essentially treating all terms from \eqref{eq:vort:aN} -- \eqref{eq:vort:cN} aside from the $\lambda$ term on the right-hand side. Notationally, we drop the subscript $N$ on $(\hat{A}_N, \hat{\omega}_N)$. We proceed in three steps, which we delineate explicitly. 

\vspace{2 mm}

\noindent \textit{Step 1: Estimate of $|\lambda| \, |\hat{A}|$.} First we have from \eqref{eq:vort:aN}: 
\begin{align} \n
|\lambda| |\hat{A}| \le & |k| |\hat{A}| + |k|^2 |\hat{A}| + |k| |\mathcal{V}[\hat{\omega}]| + |\hat{A}_{init}| \\ \label{dth2}
\lesssim & |k|^2 |\hat{A}| +|k|^{\frac23} \| (1 + y)^3 \hat{\omega}_{init} \|_{L^2_y}+  |\hat{A}_{init}|. 
\end{align} 
Above, to go from the first to second line, we have performed the following estimate:
\begin{align}\n
|\mathcal{V}[\hat{\omega}]| \le  \mathcal{V}[|\hat{\omega}|] \lesssim & |\hat{A}| \, |\lambda_{\ast}| \, \mathcal{V}[|\Omega_{BL}|] + |\hat{A}| \mathcal{V}[|\hat{\omega}_H^{(0)}|] + |\hat{A}| \mathcal{V}[|\hat{\omega}_H^{(tail)}|] + \mathcal{V}[|\hat{\omega}_{inhom}|] \\ \n
\lesssim & (1 + \frac{|k|}{\Re( \lambda)}) \frac{1}{|\lambda|^{1/2}}  |\hat{A}| +  \frac{|k|}{\Re( \lambda)} |\hat{A}| +  \frac{|k|}{\Re(\lambda) |\lambda|^{\frac12}} \Big(1 + \frac{|k|}{\Re(\lambda)}\Big) |\hat{A}| + \frac{|k|^{\frac13}}{\Re(\lambda)} \| (1 + y)^3 \hat{\omega}_{init} \|_{L^2_y} \\ 
\lesssim& |k|^{1/3} |\hat{A}| + |k|^{-\frac13} \| (1 + y)^3 \hat{\omega}_{init} \|_{L^2_y} \label{estimate_Vomega}
\end{align}
where we have used \eqref{s:a} -- \eqref{s:e}, as well as \eqref{k:hyp}-\eqref{lambda:hyp}. Plugging inequality \eqref{estimA} in the right-hand side of \eqref{dth2} and dividing by $|\lambda|$, we find
\begin{align} \n
 |\hat{A}| & \lesssim \frac{1}{|\lambda|} \Big(   k^2  \Big( \| (1 + y)^3 \hat{\omega}_{init} \|_{L^2_y} + \frac{1}{|k|^{2/3}} |\hat{A}_{init}| \Big) +  +|k|^{\frac23} \| (1 + y)^3 \hat{\omega}_{init} \|_{L^2_y}+  |\hat{A}_{init}| \Big)  \\
& \lesssim   \frac{k^2}{|\lambda|} \| (\hat{\omega}_{init}, \hat{A}_{init}) \|_H. \label{use}
\end{align}
Going back to \eqref{estimate_Vomega}, we infer 
\begin{equation} \label{estimate_Vomega2}
\mathcal{V}[|\hat{\omega}|] \lesssim \Big(  \frac{|k|^{7/3}}{|\lambda|} + k^{-\frac13} \Big) \| (\hat{\omega}_{init}, \hat{A}_{init}) \|_H \lesssim k^{5/3}  \| (\hat{\omega}_{init}, \hat{A}_{init}) \|_H
\end{equation}

\vspace{2 mm}

\noindent \textit{Step 2: Estimate of $\Re(\lambda) \|(1+y)^m\hat{\omega}\|_{L^2}$}. We now treat the quantity $\hat{\omega}$. For this, we first derive a Neumann condition for $\hat{\omega}$ by evaluating \eqref{uNeq:1} at $y= 0$, which produces
\begin{align}
\p_y \omega|_{y = 0} = \p_x |\p_x| A - \mathcal{U}[(V_{s,N} - V_s) \p_x \omega].
\end{align}
We therefore study the system 
\begin{subequations}
\begin{align}
&(\lambda + i k V_{s,N}) \hat{\omega} - i k U_s'' \mathcal{V}_y[\hat{\omega}] - \p_y^2 \hat{\omega} = i k U_s''(y) y A + \hat{\omega}_{init}, \label{subeqw} \\
&\p_y \hat{\omega}|_{y = 0} = ik|k| A - ik \mathcal{U}[(V_{s,N} - V_s) \hat{\omega}] \label{subeqw2} 
\end{align}
\end{subequations}
We take the $L^2$ (complex) scalar product of the equation with $(1+y)^{2m} \hat{\omega}$, $m=2,3$, and take the real part: 
\begin{align*}
& \Re(\lambda) \|(1+y)^{m} \hat{\omega}\|^2_{L^2} +   \| (1+y)^m  \p_y \hat{\omega} \|_{L^2}^2 - m (2m-1) \| y^{m-1} \hat{\omega} \|_{L^2_z}^2 \\ 
= \:  &   - \pa_y \hat{\omega}\vert_{y=0}  \overline{\hat{\omega}\vert_{y=0}} \: + \:  \Re \langle  i k U_s''(y)  \mathcal{V}_y[\hat{\omega}] , (1+y)^{2m} \hat{\omega} \rangle   \: + \:  \Re \langle  i k U_s''(y) y A , (1+y)^{2m} \hat{\omega} \rangle \: + \:   \Re \langle \hat{\omega}_{init} , (1+y)^{2m} \hat{\omega} \rangle \\
 \le \: &  \, |\pa_y \hat{\omega}\vert_{y=0}| \,  |\hat{\omega}\vert_{y=0}|  \: + \:    |k|  \, \mathcal{V}[|\hat{\omega}|]  \|(1+y)^{m} \hat{\omega}\|_{L^2}  
+ \:     |k| |A|  \|(1+y)^m \hat{\omega}\|_{L^2} + 
\|(1+y)^{m} \hat{\omega}_{init}\|_{L^2}   \|(1+y)^{m} \hat{\omega}\|_{L^2} 
\end{align*}
We have the inequality 
\begin{align*}
 |\pa_y \hat{\omega}\vert_{y=0}| \,  |\hat{\omega}\vert_{y=0}| & \le C  |\pa_y \hat{\omega}\vert_{y=0}|  \,   \|(1+y)^m \hat{\omega}\|_{L^2}^{1/2} \| (1+y)^m \pa_y \hat{\omega}\|_{L^2}^{1/2} \\ 
 & \le \frac{\Re(\lambda)}{2}  \|(1+y)^{m} \hat{\omega}\|^2_{L^2}  + \frac{1}{2 \Re(\lambda)}  \|(1+y)^{m}\pa_y  \hat{\omega}\|^2_{L^2}  + C' |\pa_y  \hat{\omega}\vert_{y=0}|^2 \\
 & \le \frac{\Re(\lambda)}{2}  \|(1+y)^{m} \hat{\omega}\|^2_{L^2}  + \frac{1}{2}  \|(1+y)^{m}\pa_y  \hat{\omega}\|^2_{L^2}  + C'' \Big(   k^4 |A|^2  +  k^2 |\mathcal{U}[(V_{s,N} - V_s) \hat{\omega}]|^2 \Big), 
 \end{align*}
 where the last line comes from \eqref{subeqw2}. Combining this inequality  with the usual manipulations based on Young's inequality, we end up with  
\begin{align*}
& \Re(\lambda) \|(1+y)^{m} \hat{\omega}\|^2_{L^2} +   \| (1+y)^m  \p_y \hat{\omega} \|_{L^2}^2   \\
& \lesssim \:     k^4 |A|^2  +  k^2 |\mathcal{U}[(V_{s,N} - V_s) \hat{\omega}]|^2  \:  + \:     \frac{|k|^2 \mathcal{V}[|\hat{\omega}|]^2}{\Re(\lambda)} + \frac{|k|^2 |A|^2}{\Re(\lambda)}+ \frac{\|(1+y)^{m} \hat{\omega}_{init}\|_{L^2}^2}{\Re( \lambda)} \\
& \lesssim   k^4 |A|^2 +  k^2 |\mathcal{U}[(V_{s,N} - V_s) \hat{\omega}]|^2  \:  + \:     \frac{|k|^2 \mathcal{V}[|\hat{\omega}|]^2}{\Re(\lambda)} + \frac{\|(1+y)^{m} \hat{\omega}_{init}\|_{L^2}^2}{\Re(\lambda)} 
\end{align*}
We then notice that 
\begin{align} \n
|\mathcal{U}(V_{s,N} - V_s) \omega]| \le &\int_0^{\infty} |V_{s,N} - V_s| |\omega| \le  \int_{\frac{N}{4}}^{\infty} |V_{s,N} - V_s| |\omega| \lesssim \int_{\frac{N}{4}}^{\infty} y^{-1} y^2 |\omega| \\ \label{bdUU}
\lesssim & \Big( \int_{\frac{N}{4}} y^{-2} \Big)^{\frac12} \| (1 + y)^2 \omega \|_{L^2} \lesssim \frac{1}{N^{\frac12}} \| (1 + y)^m \omega \|_{L^2}  
\end{align}
We take $N \gg |k|^{4/3}$, so that 
\begin{equation} \label{compare_N_k}
 \frac{k^2}{N}  \ll  |k|^{2/3} \lesssim \Re(\lambda)  
 \end{equation}
Combining \eqref{use}, \eqref{estimate_Vomega2}, \eqref{bdUU} and \eqref{compare_N_k}, we end up with 
\begin{align} 
\Re(\lambda) \|(1+y)^{m} \hat{\omega}\|^2_{L^2} +   \| (1+y)^m  \p_y \hat{\omega} \|_{L^2}^2   \:  & \lesssim \:     \frac{k^{8}}{|\lambda|^2} \| (\hat{\omega}_{init}, \hat{A}_{init}) \|^2_H      \:  + \:     \frac{|k|^{11/3}}{\Re(\lambda)}   \| (\hat{\omega}_{init}, \hat{A}_{init}) \|^2_H  \\
\label{estimRew}
&  \lesssim \:  \frac{k^{8}}{\Re(\lambda)}  \| (\hat{\omega}_{init}, \hat{A}_{init}) \|^2_H 
\end{align}
where the last line follows from \eqref{lambda:hyp}.  If $\Re(\lambda) \ge \frac{|\lambda|}{2}$, the bounds of the lemma follow from  \eqref{use} and \eqref{estimRew} with $m=3$: in this case, the constant $C_N$ can be taken independent of $N$. Otherwise, we move to step 3. 

\vspace{2 mm}

\noindent \textit{Step 3: Estimate of $\Im(\lambda) \|(1+y)^m\hat{\omega}\|_{L^2}$ (only needed if $|\Re(\lambda)| \le \frac{|\lambda|}{2}$}).  In this case, we have $\Im(\lambda) \ge  \frac{|\lambda|}{2}$. We take again the $L^2$ scalar product of equation \eqref{subeqw} with $(1+y)^{2m} \hat{\omega}$, but this time consider the imaginary part. There are two differences with the previous estimate for the real part: the advection term $i k V_{s,N} \hat{\omega}$ gives a non-zero contribution: 
$$ | \Im \langle i k V_{s,N} \hat{\omega} , (1+y)^{2m} \hat{\omega} \rangle | \le  |k| \| V_{s,N} (1+y)^m \hat{\omega}\|_{L^2} \, \| (1+y)^{m} \hat{\omega}\|_{L^2}  $$
Moreover, the diffusion term no longer yields a coercive term. We treat it as 
\begin{align*}
| \Im \langle \pa^2_y \hat{\omega} , (1+y)^{2m} \hat{\omega} \rangle | & \le     \|(1+y)^m \pa_y \hat{\omega}\|_{L^2}  \|(1+y)^m  \hat{\omega}\|_{L^2} + 2m (2m-1) \|(1+y)^{m-1} \hat{\omega}\|_{L^2}^2 \: + \: |\pa_y \hat{\omega}\vert_{y=0}| \, | \hat{\omega}\vert_{y=0}| \\
& \le  \frac{|\Im(\lambda)|}{2} \|(1+y)^m  \hat{\omega}\|_{L^2}^2   + \frac{C}{|\Im(\lambda)|} \|(1+y)^m\pa_y \omega\|^2 + C |\pa_y \hat{\omega}\vert_{y=0}|^2 \\
& \le \frac{|\Im(\lambda)|}{2} \|(1+y)^m  \hat{\omega}\|_{L^2}^2   + \frac{C}{|\Im(\lambda)|} \|(1+y)^m \pa_y \omega\|^2 + C'  \Big(   k^4 |A|^2  +  k^2 |\mathcal{U}[(V_{s,N} - V_s) \hat{\omega}]|^2 \Big)  \\
\end{align*}
Treating all other terms as before, we find an inequality of the type
\begin{equation} \label{eq_Imw}
\begin{aligned} 
|\Im(\lambda)| \|(1+y)^{m} \hat{\omega}\|^2_{L^2}    \: & \le \:   C \Big(  \frac{|k|^{8}}{|\Im(\lambda)|}  \| (\hat{\omega}_{init}, \hat{A}_{init}) \|^2_H   + \: \frac{1}{|\Im(\lambda)|} \|(1+y)^m \pa_y \hat{\omega}\|_{L^2}^2 \: \\
&+ \:  |k| \| V_{s,N} (1+y)^m \hat{\omega}\|_{L^2} \, \| (1+y)^{m} \hat{\omega}\|_{L^2} \Big) 
\end{aligned}
\end{equation}
 Multiplying inequality \eqref{estimRew} by $\frac{2C}{\Im(\lambda)}$  and summing it to inequality  \eqref{eq_Imw}, we end up with 
\begin{align*}
& \Big(  |\Im(\lambda)| + \frac{\Re(\lambda)}{|\Im(\lambda)|} \Big)   \|(1+y)^{m} \hat{\omega}\|^2_{L^2}  + \frac{1}{|\Im(\lambda)|} \|(1+y)^{m} \pa_y \hat{\omega}\|^2_{L^2} \\
& \lesssim \Big( \frac{k^8}{(\Re \lambda) |\Im(\lambda)|} +   \frac{k^8}{|\Im(\lambda)|} \Big) \| (\hat{\omega}_{init}, \hat{A}_{init}) \|^2_H   \: + \:  \:  |k| \| V_{s,N} (1+y)^m \hat{\omega}\|_{L^2} \, \| (1+y)^{m} \hat{\omega}\|_{L^2} 
\end{align*}
 This implies (we remind that $\Im(\lambda)| \ge \frac{|\lambda|}{2}$): 
 \begin{align*}
& |\lambda|   \|(1+y)^{m} \hat{\omega}\|^2_{L^2}   \lesssim  \frac{k^8}{|\lambda|}  \| (\hat{\omega}_{init}, \hat{A}_{init}) \|^2_H   \: + \:  \:  |k| \| V_{s,N} (1+y)^m \hat{\omega}\|_{L^2} \, \| (1+y)^{m} \hat{\omega}\|_{L^2} 
\end{align*}
To obtain the first bound of the lemma, we use the bound $|V_{s,N}| \le C N$. Hence, 
   \begin{align*}
|\lambda|   \|(1+y)^{m} \hat{\omega}\|^2_{L^2} &  \lesssim  \frac{k^8}{|\lambda|}  \| (\hat{\omega}_{init}, \hat{A}_{init}) \|^2_H   \: + \:  \:  |k| N \| (1+y)^m \hat{\omega}\|_{L^2}^2  \\
&\lesssim  \frac{k^8}{|\lambda|}   \| (\hat{\omega}_{init}, \hat{A}_{init}) \|^2_H  + |k| N \left(   \frac{k^8}{|\lambda|^2}  \| (1+y)^m \hat{\omega}\|_{L^2}^2  \: +   \:  \frac{|k| N}{\lambda} \| (1+y)^m \hat{\omega}\|_{L^2}^2 \right) \\
& \lesssim \frac{k^8}{|\lambda|}   \| (\hat{\omega}_{init}, \hat{A}_{init}) \|^2_H  + \frac{k^8 N}{|\lambda|} \| (1+y)^m \hat{\omega}\|_{L^2}^2 
\end{align*}
Note that to go from the first to the second inequality, we have plugged the first inequality in the last term $|k| N \| (1+y)^m \hat{\omega}\|_{L^2}^2$. 
But we know from \eqref{estimRew} that $\| (1+y)^m \hat{\omega}\|_{L^2}^2 \le k^8 \| (\hat{\omega}_{init}, \hat{A}_{init}) \|^2_H$, so that eventually we find 
$$ |\lambda|   \|(1+y)^{m} \hat{\omega}\|^2_{L^2} \lesssim \frac{N |k|^{16}}{|\lambda|}   \| (\hat{\omega}_{init}, \hat{A}_{init}) \|^2_H $$
Taking $m=3$, together with \eqref{use}, this yields the first bound of the lemma. As regards the second bound, we take $m=2$  and use that $|V_{s,N}(y)| \le C y$, hence: 
  \begin{align*}
|\lambda|   \|(1+y)^2 \hat{\omega}\|^2_{L^2} &  \lesssim  \frac{k^8}{|\lambda|}  \| (\hat{\omega}_{init}, \hat{A}_{init}) \|^2_H   \: + \:  \:  |k|  \| (1+y)^3 \hat{\omega}\|_{L^2}  \| (1+y)^2 \hat{\omega}\|_{L^2}
\end{align*}
By Young's inequality,   
 \begin{align*}
|\lambda|   \|(1+y)^2 \hat{\omega}\|^2_{L^2} &  \lesssim  \frac{k^8}{|\lambda|}  \| (\hat{\omega}_{init}, \hat{A}_{init}) \|^2_H   \: + \:  \:  \frac{|k|^2}{|\lambda|}  \| (1+y)^3 \hat{\omega}\|_{L^2}^2
\end{align*}
But from \eqref{estimRew} applied with $m=3$, we know that $\| (1+y)^3 \hat{\omega}\|_{L^2}^2 \le k^8 \| (\hat{\omega}_{init}, \hat{A}_{init}) \|^2_H$, hence 
 \begin{align*}
|\lambda|   \|(1+y)^2 \hat{\omega}\|^2_{L^2} &  \lesssim  \frac{k^{10}}{|\lambda|}  \| (\hat{\omega}_{init}, \hat{A}_{init}) \|^2_H 
\end{align*}
Together with \eqref{use}, this yields the second bound of the lemma, and concludes the proof. 
\end{proof}
We can now prove our main theorem, Theorem \ref{thm:main}. 

\begin{proof}[Proof of Theorem \ref{thm:main}] As discussed at the beginning of Paragraph \ref{subsec:decompo}, it is enough to show that for any initial data $\big(\omega_{init}, A_{init} = \mathcal{U}[\omega_{init}] \big)$ 
satisfying 
$$  \| e^{c_0 |\pa_x|^{2/3}} (1+y)^3 \omega_{init}\|_{L^2(\R \times \R_+)}  < +\infty, \quad c_0 > 0, $$
there exists  $\beta, C, s > 0$, such that system \eqref{vbo} has a unique solution in $[0, T= \frac{c_0}{\beta}[$ satisfying
\begin{equation} \label{stab_estimate_Gevrey}
\|  e^{(c_0 - \beta t) |\pa_x|^{2/3}} (1+y)^2 \omega(t, \cdot)\|_{L^2(\R \times \R_+)} \le C  \| e^{c_0 |\pa_x|^{2/3}} (1+|\pa_x|)^s  (1+y)^3 \omega_{init}\|_{L^2(\R \times \R_+)} 
\end{equation}
Going to Fourier in $x$,  it is enough to show that for all $k \in \R^*$, and all data $\big( \hat{\omega}_{init}, A_{init} =\mathcal{U}[\hat{\omega}_{init}]   \big) \in H$, system 
\begin{subequations} \label{vbo_Fourier}
\begin{align}
&\p_t \hat{A} + i k  \hat{A} + i k |k| \hat{A} = i k  \mathcal{V}[\hat{\omega} ], \\ 
 &\p_t \hat{\omega} + i k V_s  \hat{\omega}  - i k  U_s'' \mathcal{V}_y[\hat{\omega} ] - \p_y^2 \hat{\omega} = i k y U_s'' \hat{A}, \\
 & \mathcal{U}[\hat{\omega}] = \hat{A}
\end{align}
\end{subequations}
has a global in time solution $\big( \hat{\omega}, \hat{A} \big)$ starting from $\big( \hat{\omega}_{init}, A_{init} \big)$, and satisfying:
\begin{equation} \label{target_estimate}
\| ((1+y)^{-1}\hat{\omega}(t,\cdot), \hat{A}(t))\|_H  \le C e^{\beta |k|^{2/3} t} (1+|k|)^s   \| \big( \hat{\omega}_{init}, A_{init} \big)  \|_H 
\end{equation}
Indeed, as $|A_{init}| \lesssim \|(1+y)^3 \hat{\omega}_{init}\|_{L^2(\R_+)}$, this implies 
\begin{equation*}
 \|(1+y)^2 \hat{\omega}(t, \cdot) \|_{L^2(\R_+)} \le C' e^{\beta |k|^{2/3} t} (1+|k|)^s  \|(1+y)^3 \hat{\omega}_{init}\|_{L^2(\R_+)} 
\end{equation*}
Multiplying each side by $e^{(c_0 - \beta t) |k|^{2/3}}$, squaring, integrating in $k$ and using Plancherel theorem, we find   \eqref{stab_estimate_Gevrey}. 

We first consider the case of low frequencies, namely  $|k| \le k_0$, where $k_0$ was introduced in Proposition \ref{pro:main}.  In this case, we use a fact emphasized in Paragraph \ref{subsec:decompo}:  solving \eqref{vbo_Fourier} under the condition $\hat{A} = \mathcal{U}[\hat{\omega}]$ is equivalent to solving it under the Neumann condition 
$$ \pa_y \hat{\omega}\vert_{y=0} = i k |k| \hat{A} $$
Under this more standard condition, solving system \eqref{vbo} for fixed $k$ is easy. Namely, by lifting the inhomogenous boundary data and using classical weighted $L^2$ estimates,  one can construct a unique global solution in $C(\R_+, H)$ satisfying  
\begin{equation}
\frac{d}{dt} \|\big( \hat{\omega}, \hat{A} \big) \|^2_{H} + \| (1+y)^2 \pa_y \omega\|_{L^2}^2  \lesssim  |k|^3 \|\big( \hat{\omega}, \hat{A} \big) \|^2_{H} 
\end{equation}
It follows that: 
\begin{equation} \label{main:bound}
\| (\hat{\omega}(t,\cdot), \hat{A}(t))\|_H  \le C e^{C |k|^{3} t}   \| \big( \hat{\omega}_{init}, A_{init} \big)  \|_H 
\end{equation}
This inequality implies \eqref{target_estimate} for $|k| \le k_0$, with $s=0$, $\beta = C k_0^{7/3}$. Let us mention briefly that a similar standard Gronwall estimate (with bad growth rate $|k|^3$) could have been established starting from the equivalent velocity formulation  of \eqref{vbo_Fourier}, that is in terms of $(\hat{u} = \int_0^y \hat{\omega}, \hat{A})$ rather than in terms of $(\hat{\omega}, \hat{A})$. In particular, thanks to this velocity estimate, one can check that $(\hat{\omega}, \hat{A})$ is unique among all solutions in $L^\infty_{loc}(\R_+, L^2(\R_+)) \times L^\infty_{loc}(\R_+)$  (without asking for regularity of $\pa_y \hat{\omega}$). 

Hence, the last point is to show that such solution satisfies \eqref{target_estimate} in the high frequency regime $|k| \ge k_0$. We will prove this by compactness, through consideration of the approximate systems: 
\begin{subequations} \label{vbo_FourierN}
\begin{align}
&\p_t \hat{A}_N + i k  \hat{A}_N + i k |k| \hat{A}_N = i k  \mathcal{V}[\hat{\omega} ], \\ 
 &\p_t \hat{\omega}_N + i k V_{s,N}  \hat{\omega}_N  - i k  U_s'' \mathcal{V}_y[\hat{\omega}_N ] - \p_y^2 \hat{\omega}_N = i k y U_s'' \hat{A}_N, \\
 & \mathcal{U}[\hat{\omega}_N] = \hat{A}_N
\end{align}
\end{subequations}
where $V_{s,N}$ was defined in \eqref{approximate_shear_flow}. System \eqref{vbo_FourierN} can be written in an abstract way as: 
$$ \pa_t \big(  \hat{\omega}_N , \hat{A}_N  \big) + L_{k,N} \big( \hat{\omega}_N , \hat{A}_N  \big) = 0, $$
where we see $L_{k,N}$ as the (closed, densely defined) linear operator from 
$$ D(L_{k,N}) \: := \: \Big\{ (\hat{\omega}, \hat{A}) \in H_{\mathcal{U}}, \quad  \pa^2_y  \hat{\omega} \in L^2((1+y)^2 dy) \Big\} \quad \text{into} \quad  H_{\mathcal{U}}   \: := \: \Big\{ (\hat{\omega}, \hat{A}) \in H, \: \mathcal{U}[\hat{\omega}] =  \hat{A}  \Big\} $$
It follows from the resolvent estimate in Lemma \ref{lem:resolvent_estim} that the operator $L_{k,N}$ is sectorial.  More precisely, let $\theta_{k,N} := \frac{1}{2 C_N} |k|^{-s_0}$. Taking $|k| \ge k_0$ large enough, we can always assume $\theta_{k,N} \le \frac{1}{2}$. For any $\lambda_0 \in \{\Re(\lambda) = K_* |k|^{2/3} \}$, and any $\lambda \in \overline{D(\lambda_0, \theta_{k,N} |\lambda_0|)}$, $\lambda + L_{k,N}$ is invertible:  indeed,  
$$  (\lambda \textrm{Id} + L_{k,N}) =  (\lambda_0 \textrm{Id} + L_{k,N})  \Big(   (\lambda - \lambda_0)  (\lambda_0 \textrm{Id} + L_{k,N})^{-1} + \textrm{Id} \Big)  $$
and the last factor at   the right-hand side has norm 
$$|\lambda - \lambda_0| \|  (\textrm{Id} + L_{k,N})^{-1} \|_{H \rightarrow H} \le \theta_{k,N} C_N |k|^{s_0} \le \frac{1}{2}. $$
where the first inequality comes from Lemma \ref{lem:resolvent_estim}. Moreover, 
\begin{equation}
 \|(\lambda \textrm{Id} + L_{k,N})^{-1}\|_{H_\mathcal{U} \rightarrow H_\mathcal{U}} \le  2 \| (\lambda_0 \textrm{Id} + L_{k,N})^{-1}\|_{H \rightarrow H} \le \frac{2 C_N |k|^{s_0}}{|\lambda_0|} \le \frac{4 C_N |k|^{s_0}}{|\lambda|} 
\end{equation}
In particular, the resolvent set of $L_k$ contains 
$$ \bigcup_{\lambda_0 \in  \{\Re(\lambda) = K_* |k|^{2/3} \}}  \overline{D(\lambda_0 \theta_{k,N})}  \: \supset \:  \Gamma_{k,N} :=  \Big\{ \Re(\lambda) = - \theta_{k,N} |\Im(\lambda)| + K_* |k|^{2/3} \Big\}. $$
By standard results for sectorial operators, $L_{k,N}$ generates an analytic semigroup that can be written as 
$$ e^{-t L_{k,N}} = \frac{1}{2i \pi} \int_{\Gamma_{k,N}} e^{\lambda t} (\textrm{Id} + L_{k,N})^{-1} d \lambda$$
resulting in 
$$  \| e^{-t L_{k,N}} \|_{H \rightarrow H} \lesssim  e^{K_* |k|^{2/3} t} |k|^{s_0} \int_{\R} e^{-\theta_{k,N} |y|} \frac{1}{1+|y|} dy \le 4 C_N e^{K_* |k|^{2/3} t} |k|^{2 s_0} $$
It implies 
\begin{equation} \label{target_estimateN}
\| (\hat{\omega}_N(t,\cdot), \hat{A}_N(t,\cdot))\|_H  \le 4 C_N e^{K_* |k|^{2/3} t} (1+|k|)^{2 s_0}   \| \big( \hat{\omega}_{init}, A_{init} \big)  \|_H 
\end{equation}
 This is unfortunately still not enough, as the constant $C_N$  may go to infinity with $N$.  To obtain a  uniform bound, we proceed as follows. We introduce a lift of the initial condition, namely the couple  
 $$ \big( \hat{\omega}_{lift}, \hat{A}_{lift} \big) =  e^{-t}\big( \hat{\omega}_{heat}, A_{init} \big)$$
  where   $\hat{\omega}_{heat}$ satisfies 
 $$ \pa_t \hat{\omega}_{heat} - \pa^2_y \hat{\omega}_{heat} = 0, \quad \pa_y \hat{\omega}_{heat} = 0, \quad \hat{\omega}_{heat}\vert_{t=0} = \hat{\omega}_{init}. $$
Integrating the equation in $y$, one can check that $\mathcal{U}[\hat{\omega}_{heat}] =  A_{init}$.  We then set 
\begin{align}
W_N & = e^{-2 K_\ast |k|^{2/3}t} \left( \hat{\omega}_N -  \hat{\omega}_{lift} \right) \quad \text{for } \: t > 0, \quad W_N = 0  \quad \text{for } \: t \le 0 \\
B_N  & = e^{-2 K_\ast |k|^{2/3}t} \left( \hat{A}_N -  \hat{A}_{lift} \right)\quad \text{for } \: t > 0, \quad B_N = 0  \quad \text{for } \: t \le 0
\end{align}
From \eqref{target_estimateN}, we know that $(W_N, B_N)$ belongs to $L^2(\R_t, H)$. Moreover, it satisfies for all $t \in \R$, $y \in \R_+$:
\begin{subequations} \label{vbo_FourierBN}
\begin{align}
&(2 K_\ast |k|^{2/3} + \p_t ) B_N + i k  B_N + i k |k| B_N = i k  \mathcal{V}[W_N] +  f \\ 
 &(2 K_\ast |k|^{2/3} + \p_t ) W_N + i k V_{s,N}  W_N  - i k  U_s'' \mathcal{V}_y[W_N ] - \p_y^2 W_N = i k y U_s'' B_N + F, \\
 & \mathcal{U}[W_N] = B_N
\end{align}
\end{subequations}
One can check that $f = f(t)$ and $F = F(t,y)$, which are expressed in terms of $V_{s,N}$, $k$, and $\big( \hat{\omega}_{lift}, \hat{A}_{lift} \big)$ satisfy 
$$ \| (f, F) \|_{L^2(\R_t, H)} \lesssim |k|^2 \| (\hat{\omega}_{init}, A_{init}) \|_{H}. $$
 Taking the Fourier transform in time, with $\tau$ the dual variable of $t$, we end up with the system 
$$ \Big( \lambda I + L_{k,N} \Big) (\hat{W}_N, \hat{B}_N) = (\hat{F}, \hat{f}), \quad \text{ with } \lambda := 2 K_\ast |k|^{2/3} + i \tau. $$
 We use this time the second resolvent estimate of Lemma \ref{lem:resolvent_estim}, to find 
 $$ \| ((1+y)^{-1} \hat{W}_N(\tau,\cdot), \hat{B}_N(\tau) \|_{H}  \le C_0 \frac{|k|^{s_0}}{|\lambda|} \| (\hat{F}(\tau, \cdot) , \hat{f}(\tau))\|_H $$
The right-hand side is integrable in $\tau$, as the product of two $L^2$ functions.  Applying the inverse Fourier transform in $\tau$ and Cauchy-Schwarz, we deduce a pointwise in time bounds, namely,
\begin{align*}
  \| ((1+y)^{-1} W_N(t,\cdot), B_N(t) \|_{H} & \le C |k|^{s_0} \| (\hat{F} , \hat{f})\|_L^2(\R_\tau, H)  = \frac{C}{2\pi} |k|^{s_0} \| (F , f)\|_{L^2(\R_t, H)}  \\
  & \le C |k|^{s_0+2}  \| (\hat{\omega}_{init}, A_{init}) \|_{H} 
\end{align*}
Estimate 
\begin{equation*} 
\| ((1+y)^{-1}\hat{\omega}_N(t,\cdot), \hat{A}_N(t,\cdot))\|_H  \le C_0 e^{\beta |k|^{2/3} t} (1+|k|)^s   \| \big( \hat{\omega}_{init}, A_{init} \big)  \|_H 
\end{equation*}
with $\beta = 2 K_\ast$, $s = s_0 + 2$, follows. One can then send $N$ to infinity and obtain estimate \eqref{target_estimate} as expected. This concludes the proof. 
\end{proof}

\section{Linear hydrostatic Navier-Stokes} \label{sec:HNS}
We explain in this section how to adapt the analysis of \eqref{ltd:1}-\eqref{ltd:2} to prove Theorem \ref{thm:main2}. 
Differentiating the first equation \eqref{LHNS} with respect to $y$, we find, for $\omega = \pa_y u$: 
\begin{equation*}
\begin{aligned}
& \p_t \omega + U_s \p_x \omega - U_s'' v - \p_y^2 \omega = 0,   \\
& \pa_x u + \pa_y v = 0, \\
& u\vert_{y=0,1} = v\vert_{y=0,1}. 
\end{aligned}
\end{equation*}
Inspired by the previous sections, we could try to rely on  a similar iteration scheme (at the level of the resolvent equation):  at each step we  would solve the equation on vorticity with an artificial homogeneous Neumann condition, and then rectify boundary conditions on $u$ and $v$. However, correcting the boundary condition on $v$ would generate error terms that have too large amplitude. More precisely, in the analogue of equation \eqref{gtgt}, the analogue of the source term $\mathcal{V}_y[\Omega_{BL}]$ would be too large, of size $1$ rather than $|\lambda|^{-1/2}$. This would prevent the convergence of the series. Therefore, we have to change the iteration, in such a way that at each step the homogeneous Dirichlet  condition on $v$ is maintained. In particular, we do not want to recover $v$ from $\omega$ using the formula $v = -  \int_0^y \int_1^{y'} \pa_x  \omega$, because given an arbitrary function $\omega$, it does not necessarily vanish at $y=1$. This implies not to use the exact analogue of operators $\mathcal{U}$ and $\mathcal{V}_y$ introduced in the Triple-Deck analysis.  Following more closely the approach in \cite{GVMM}, we first introduce the stream function $\Phi[\omega] =\Phi[\omega](y)$ defined as the solution of the Dirichlet problem  
\begin{equation} \label{def:Phi}
 \pa^2_y \Phi[\omega] = \omega, \quad \Phi[\omega]\vert_{y=0,1} = 0. 
 \end{equation}
so that 
$$ u = \pa_y \Phi[\omega], \quad  v = - \Phi[\omega_x].  $$ 
We then define 
\begin{equation} \label{defi:U} 
\mathcal{U}[\omega] = \big( \pa_y \Phi[\omega](0), \pa_y \Phi[\omega](1) \big)^t 
\end{equation}
Finally, \eqref{LHNS} is easily shown to be equivalent to: 
\begin{subequations} \label{LHNSom}
\begin{align}
& \p_t \omega + U_s \p_x \omega - U_s'' \Phi[\omega_x] - \p_y^2 \omega = 0,   \label{LHNSoma} \\
&  \mathcal{U}[\omega] = 0 \label{LHNSomb}
\end{align}
\end{subequations}
where the relation in  \eqref{LHNSomb} corresponds to the Dirichlet conditions on $u$, while the Dirichlet conditions on $v$ are automatically encoded in the definition \eqref{def:Phi} of $\Phi[\omega]$. Note also that differentiating in $x$ and  integrating in $y$ the first equation  of \eqref{LHNS}, we find 
$$ - \pa_{xx} p =  2 \pa_{xx}   \int_0^1  \big(U_s \pa_y \Phi[\omega]\big) dy + \pa_x \omega\vert_{y=0} -  \pa_x  \omega\vert_{y=1}  $$
 so that 
 $$ \pa_x p =  -2 \pa_x \int_0^1 \big(U_s \pa_y \Phi[\omega]\big) dy   +  \omega\vert_{y=1} -    \omega\vert_{y=0}. $$
  Eventually, evaluating the first equation of \eqref{LHNS} at $y=0,1$ yields the mixed type boundary condition: 
  \begin{equation} \label{neumann_bc_LHNS}
  \pa_y \omega\vert_{y=0,1} =  -2 \pa_x \int_0^1  \big(U_s \pa_y \Phi[\omega]\big) dy  +  \omega\vert_{y=1} -    \omega\vert_{y=0}. 
  \end{equation}
Similarly to the case of the Triple-Deck model, one can show that solving \eqref{LHNSoma} under condition \eqref{LHNSomb} is the same as solving it under \eqref{neumann_bc_LHNS}. 

The main ingredient to prove the  Gevrey $3/2$ well-posedness of system \eqref{LHNSom} is again a stability estimate  for the resolvent equation 

\begin{equation} \label{resolvent_LHNS}
\lambda \hat{\omega} + i k  U_s \hat{\omega} - i k U_s'' \Phi[\hat{\omega}] - \p_y^2 \hat{\omega} = \hat{\omega}_{init} 
\end{equation}
where $\lambda \in \mathbb{C}$, $k \in \R^*$, and  $\hat{\omega}_{init} = \hat{\omega}_{init}(y)$ belongs to  the space 
$$ H = \big\{ \hat{\omega} \in L^2((0,1)), \quad \mathcal{U}[\hat{\omega}] =  0\big\} $$ 
equipped with the $L^2$ norm. Namely, we have 
 \begin{pro} \label{pro:main2HNS}
 There exist absolute positive constants $K_{\ast},  k_0$ and $M$, such that for all $|k| \ge k_0$,  all $\lambda$ with $\Re(\lambda) \ge K_{\ast} |k|^{2/3}$, and all data $\hat{\omega}_{init}  \in H$,  equation \eqref{resolvent_LHNS} has a unique solution $\hat{\omega}$  satisfying 
$$  \|\hat{\omega}\|_{L^2}  \lesssim |\lambda|^{1/4} |k|^{-2/3} \|\hat{\omega}_{init} \|_{L^2}. $$
 \end{pro}
On the basis of this proposition, by the same kind of reasoning as in Section \ref{sec:proof:thm}, one proves Theorem \ref{thm:main2}. Actually, the reasoning of Section \ref{sec:proof:thm} can be greatly simplified in this case: the $y$-domain being $(0,1)$ instead of $\R_+$,  there is no difficulty related to the unboundedness of the advection field $V_s$: one can prove directly sectoriality on the original operator, without any approximation. For brevity, we do not give further details for this last part, and just explain how to prove Proposition \ref{pro:main2HNS}. 

\subsection{Iteration scheme}
Similarly to Section \ref{subsec:iteration}, the idea is to look for a solution of \eqref{resolvent_LHNS} under the form of a series made of hydrostatic and boundary layer terms: 
\begin{align} 
\hat{\omega} = & \omega_H^{(0)}  + \omega_{BL}^{(0)} + \sum_{j = 1}^{\infty} \omega_H^{(j)} + \sum_{j = 1}^{\infty} \omega_{BL}^{(j)} 
\end{align}
Again, we  initialize the construction by solving the Neumann problem: 
\begin{equation} \label{eq:vort:bar:a:HNS}
\begin{aligned}
& (\lambda + i k U_s) \omega_H^{(0)}   - ik U''_s \Phi[\omega_H^{(0)}]   - \pa^2_y \omega_H^{(0)}  = \hat{\omega}_{init}, \\
& \p_y \omega_H^{(0)}|_{y = 0} = 0. 
\end{aligned}
\end{equation}
We then  initialize the boundary layer construction by solving the system: 
\begin{equation} \label{eq:vort:bar:BL:A:HNS}
\begin{aligned}
& (\lambda + i k U_s) \omega_{BL}^{(0)}   - \pa^2_y \omega_{BL}^{(0)}  = 0, \\
& \mathcal{U}[ \omega_{BL}^{(0)}] =  - \mathcal{U}[ \omega_H^{(0)}],
\end{aligned}
\end{equation}
where we remind that this time, the operator $\mathcal{U}$ is defined by \eqref{defi:U}, and involves  the streamfunction  $\Phi[\omega]$ defined in \eqref{def:Phi}.  Construction of solutions to \eqref{eq:vort:bar:a:HNS} and \eqref{eq:vort:bar:BL:A:HNS} will be discussed below. Note that we get again rid of the stretching term in \eqref{eq:vort:bar:BL:A:HNS}.  This creates an error term   $- i k U_s'' \Phi[\omega_{BL}^{(0)} ]$, which will be corrected by the next hydrostatic term in the expansion: more generally for $j \ge 1$, we introduce the solution $\omega_H^{(j)}$ of 
\begin{equation} \label{eq:vort:bar:ajHNS}
\begin{aligned}
& (\lambda + i k U_s) \omega_H^{(j)}   - ik U''_s \Phi[ \omega_H^{(j)}]   - \pa^2_y \omega_H^{(j)}   =   i k U_s'' \Phi[ \omega_{BL}^{(j-1)} ] \\
& \p_y \omega_H^{(j)}|_{y = 0} = 0, 
\end{aligned}
\end{equation}
and the solution  $\omega_{BL}^{(j)}$ of 
\begin{equation} \label{eq:vort:bar:BL:AjHNS}
\begin{aligned}
& (\lambda + ik U_s)\omega_{BL}^{(j)}  - \pa^2_y \omega_{BL}^{(j)}  = 0, \\
&\mathcal{U}[ \omega_{BL}^{(j)}] =  - \mathcal{U}[ \omega_H^{(j)}].
\end{aligned}
\end{equation}
Similarly to Paragraph \ref{subsec:iteration}, one can simplify the expressions for  $(\omega_H^{(j)},\omega_{BL}^{(j)})$. We first introduce the sequence of vectors in $\R^2$:  
\begin{align} \label{hg1HNS}
\lambda_j :=  &  -\mathcal{U}[\omega_H^{(j)}]     
\end{align}
as well as the vector-valued function $\Omega_{BL} = \Omega_{BL}(y) \in \R^2$  defined by : for all vector $\Lambda  \in \R^2$,  $\Omega_{BL} \cdot \Lambda$ satisfies the system
\begin{equation} \label{gtgt2HNS}
\begin{aligned}
& (\lambda + ik U_s) \big( \Omega_{BL} \cdot \Lambda \big)  - \pa_y^2  \big( \Omega_{BL} \cdot \Lambda \big)  = 0, \\
& \mathcal{U}[  \Omega_{BL} \cdot \Lambda ]  = \Lambda
\end{aligned}
\end{equation}
Finally, we introduce  the vector-valued function  $F_H = F_H(y) \in \R^2$ of  
\begin{equation} \label{gtgtHNS}
\begin{aligned}
& (\lambda + i k U_s) F_H   - ik U''_s \Phi[ F_H]   - \pa^2_y F_H  =  U_s'' \Phi[\Omega_{BL}], \\
& \p_y F_H|_{y = 0} = 0, 
\end{aligned}
\end{equation}
Anticipating that these functions are well-defined for $\Re(\lambda) \ge K_\ast |k|^{2/3}$, it follows that 
\begin{align*} 
\omega^{(j)}_{BL} = &  \Omega_{BL} \cdot \lambda_j, \qquad j \ge 0, \\ 
\omega_H^{(j)} = & ik F_H \cdot \lambda_{j-1}, \qquad j \ge 1.
\end{align*}
and inserting this last relation into the  formula for $\lambda_j$, we find 
\begin{align*} 
 \lambda_{j+1} & = - i k \mathcal{U}[F_H \cdot \lambda_j],  
 \end{align*}
that is: for all $j \ge 0$, 
 \begin{equation} \label{defi_MH}
  \lambda_j = (- i k M_H)^j \lambda_0, \quad M_H \Lambda :=  \mathcal{U}[F_H \cdot \Lambda].
 \end{equation}
 It remains to show the well-posedness of the boundary layer system \eqref{gtgt2HNS}, the hydrostatic systems \eqref{eq:vort:bar:BL:A:HNS} and  \eqref{gtgtHNS}, and finally show that  the matrix $M_H$ satisfies $|ikM_H| \ll 1$, so that the series defining $\hat{\omega}$ will converge. Again, this will be possible under conditions \eqref{k:hyp}-\eqref{lambda:hyp}.

\subsection{Construction and convergence of the iteration}

\subsubsection{Boundary layer part} 
 The only significant change compared to the analysis of the previous sections is the treatment of the boundary layer model \eqref{gtgt2HNS}, as the operator $\mathcal{U}$ is now defined in terms of the stream function. By linearity with respect to $\lambda$, the function $\Omega_{BL}$ (with values in $\R^2$)  solves 
  \begin{equation} \label{gtgt2_bis_HNS}
\begin{aligned}
& (\lambda + ik U_s)  \Omega_{BL}  - \pa_y^2   \Omega_{BL}   = 0, \\
& \mathcal{U}[\Omega_{BL}]  = \textrm{Id}.
\end{aligned}
\end{equation}
where $\mathcal{U}[\Omega_{BL}]$ is a $2 \times 2$ matrix: more generally, for any function $\Omega$ with values in $\R^2$, $\mathcal{U}[\Omega]$ is defined by 
$$ \mathcal{U}[\Omega] \Lambda = \mathcal{U}[\Omega \cdot \Lambda], \quad \forall \Lambda \in \R^2. $$
 As in Section \ref{sec:OmegaBL}, we look for this solution under the form 
 \begin{align} \label{blexpHNS} 
\Omega_{BL} := \sum_{j = 0}^{\infty} (\xi^{(j)} + \Xi^{(j)}),
\end{align}
where  $\xi^{(j)}$, $\Xi^{(j)}$  have values in $\R^2$ and solve the systems: 
\begin{subequations}
\begin{align} \label{m:11HNS}
&\lambda \xi^{(j)} - \p_y^2 \xi^{(j)} = 0, \\ \label{m:12HNS}
&\mathcal{U}[\xi^{(0)}] = \textrm{Id},  \\
& \mathcal{U}[\xi^{(j)}] = - \mathcal{U}[\Xi^{(j-1)}] \quad \text{ for $j \ge 1$}
\end{align}
\end{subequations}
while 
\begin{equation} \label{eq:XijHNS}
\begin{aligned}
& (\lambda + ik U_s)\Xi^{(j)}  - \pa_y^2 \Xi^{(j)} =  - i k U_s \xi^{(j)}, \\
& \Xi^{(j)}|_{y = 0} = 0, \quad \text{ for $j \ge 0$}
\end{aligned}
\end{equation}
Still following Section \ref{sec:OmegaBL}, defining the matrix 
$$ \alpha_j := \mathcal{U}[\Xi^{(j)}],$$ 
one has: 
\begin{equation*}
\xi^{(j)} = - \xi^{(0)} \alpha_{j-1}, \quad \Xi^{(j)} = - \Xi^{(0)} \alpha_{j-1} 
\end{equation*}
resulting in 
$$ \alpha_j = - \alpha_0 \alpha_{j-1}, \quad j \ge 1, \quad \alpha_0 = \mathcal{U}[\Xi^{(0)}]. $$
The point is to construct $\xi^{(0)}$, $\Xi^{(0)}$, and show that the matrix $\alpha_0$ has norm strictly less than $1$. As regards $\xi^{(0)}$, it is better to reformulate \eqref{m:11HNS} in terms of the stream function $\Phi^{(0)} := \Phi[\xi^{(0)}]$, that satisfies 
\begin{subequations}
\begin{align} 
&\lambda \pa^2_y \Phi^{(0)}   - \p_y^4 \Phi^{(0)}  = 0, \\ 
& \Phi^{(0)}\vert_{y=0,1} = 0, \quad \pa_y \Phi^{(0)}\vert_{y=0} = (1,0)^t, \quad \pa_y \Phi^{(0)}\vert_{y=1} = (0,1)^t
\end{align}
\end{subequations}
This can be solved explicitly: one has 
$$ \Phi^{(0)} =   \left(  \begin{smallmatrix} a^-_\lambda \\ b^-_{\lambda} \end{smallmatrix} \right) e^{-\lambda^{1/2} y} 
 +   \left(  \begin{smallmatrix} - b^-_\lambda \\ -a^-_{\lambda} \end{smallmatrix} \right) e^{-\lambda^{1/2} (1-y)} +   \left(  \begin{smallmatrix} c_\lambda \\ c_{\lambda} \end{smallmatrix} \right) (y - \frac{1}{2})  +   \left(  \begin{smallmatrix} - d_\lambda \\ d_{\lambda} \end{smallmatrix} \right). $$
 where 
 $$ a^-_\lambda  \sim -\lambda^{-1/2}, \quad b^-_\lambda  \sim - \lambda^{-1},   \quad c_\lambda \sim -  \lambda^{-1/2}, \quad d_\lambda \sim -\frac{\lambda^{-1/2}}{2}. $$ 
as $|\lambda| \rightarrow +\infty$, which is the asymptotics relevant to the regime \eqref{k:hyp}-\eqref{lambda:hyp}. This implies 
$$ \xi^{(0)} =  \xi^{(0,-)}  + \xi^{(0,+)}, \quad    \xi^{(0,-)} :=  \lambda  \left(  \begin{smallmatrix} a^-_\lambda \\ b^-_{\lambda} \end{smallmatrix} \right) e^{-\lambda^{1/2} y}, \quad   \xi^{(0,+)} :=   \lambda  \left(  \begin{smallmatrix} - b^-_\lambda \\ -a^-_{\lambda} \end{smallmatrix} \right) e^{-\lambda^{1/2} (1-y)}. $$
Note that  $\xi^{(0,-)}$ is localized at scale $|\lambda|^{-1/2}$ near $y=0$, while  $\xi^{(0,+)}$ is localized at scale $|\lambda|^{-1/2}$ near $y=1$. By a straightforward adaptation of Lemma \ref{lem_xi0}, we obtain
\begin{lem} \label{lem_xi0pm}
System \eqref{eq:XijHNS} with $j=0$  has a unique solution  $\Xi^{(0)} = \Xi^{(0,-)} + \Xi^{(0,+)} $  satisfying for all $m \ge 0$:   
\begin{align} 
 \| y^m \Xi^{(0,-)}  \|_{L^2}^2 \lesssim_m  \frac{k^2}{ |\lambda|^{m + 5/2}}, \quad  \| y^m \p_y \Xi^{(0,-)}  \|_{L^2}^2 \lesssim_m  \frac{k^2}{ |\lambda|^{m + 3/2}}.\\
  \| (1-y)^m \Xi^{(0,+)}  \|_{L^2}^2 \lesssim_m  \frac{k^2}{ |\lambda|^{m + 5/2}}, \quad  \| (1-y)^m \p_y \Xi^{(0,+)}  \|_{L^2}^2 \lesssim_m  \frac{k^2}{ |\lambda|^{m + 3/2}}
\end{align}
where the implicit constant in the above inequalities depends on $m$. 
\end{lem}
From there, one can have bounds on $\Phi[\Xi^{(0)}]$ and $\mathcal{U}[\Xi^{(0)}]= \big( \pa_y  \Phi[\Xi^{(0)}](0), \pa_y  \Phi[\Xi^{(0)}](1) \big)$. From the representation formula 
$$ \Phi[\Xi^{(0,+)}](y) = \int_0^y (y-1) y' \Xi^{(0,+)}(y') dy' + \int_y^1 (y'-1) y \Xi^{(0,+)}(y') dy'  $$
we deduce that 
\begin{align*}
| \Phi[\Xi^{(0,+)}](y) | \le 2 \int_{0}^1 |y'  \Xi^{(0,+)}(y')| dy' \lesssim |\lambda|^{-1/2} 
\end{align*}
where the last inequality is obtained as in  \eqref{boundVXi0}, thanks to Lemma \ref{lem_xi0pm}.  The same holds symmetrically for  $\Phi[\Xi^{(0,-)}](y)$, and so 
$$ \sup_{y} |\Phi[\Xi^{(0)}](y) | \lesssim  |\lambda|^{-1/2}.  $$
Also, we find that 
$$ |\pa_y \Phi[\Xi^{(0,\pm)}](y)| \le \int_0^1  |\Xi^{(0,\pm)}(y')| dy'  \lesssim |k|  |\lambda|^{-3/2} $$
where the last inequality is obtained as in \eqref{estimalpha0}, thanks to Lemma \ref{lem_xi0pm}. It follows that 
$$ \mathcal{U}[\xi^0] \lesssim |k| |\lambda|^{-3/2} $$
On the basis of all these bounds, one has easily the following analogue  of Corollary  \ref{corXi0}:  
\begin{cor}  \label{corXi0HNS}
Under assumptions \eqref{k:hyp}-\eqref{lambda:hyp},  the constant $\alpha_0 = \mathcal{U}[\Xi^{(0)}]$ satisfies  
\begin{align} \label{alpha0bdHNS}
|\alpha_0| <& 1 
\end{align}
As a consequence, the  function $\Omega_{BL}$  introduced in \eqref{blexpHNS} is well-defined in $H^1(0,1)$, and is a solution of \eqref{gtgt2HNS}. Moreover, it satisfies the estimate: 
\begin{align}
\sup_{y \ge 0} |\Phi[\Omega_{BL}](y)| \lesssim & \frac{1}{|\lambda|^{\frac12}}.
\end{align}
\end{cor} 

\subsubsection{Hydrostatic part}
The construction of the hydrostatic terms $\omega^{(0)}_H$ and $F_H$, solving \eqref{eq:vort:bar:a:HNS} and \eqref{gtgtHNS}, is based as in Section \ref{sec:OmegaH} on the use of weighted norms $\| \omega/(-U''_s)^{1/2}\|_{L^2}$. This is actually simpler, as we do not have problems related to decay at infinity: we can use directly $U''_s$ in the  weight, instead of $U''_{s,k}$. From there, the estimates
\begin{align} \label{bdH1HNS}
\Re(\lambda) \| \frac{1}{(- U_{s}'')^{1/2}} F_H \|_{L^2}^2 + \|\frac{1}{(- U_{s}'')^{1/2}} \p_y F_H \|_{L^2}^2  & \lesssim  \frac{1}{\Re(\lambda) |\lambda|}\\ 
\label{bdIH0HNS}
\Re(\lambda) \| \frac{1}{(- U_{s}'')^{1/2}} \omega_{H}^{(0)} \|_{L^2}^2 + \|\frac{1}{(- U_{s}'')^{1/2}} \p_y \omega_{H}^{(0)} \|_{L^2}^2 \lesssim & \frac{ 1 }{\Re(\lambda)} \| \frac{1}{(- U_{s}'')^{1/2}} \omega_{init} \|_{L^2_y}^2 
\end{align}
are proved as the ones of $F_H$ and $\omega^{(0)}_{IH}$ in section \ref{sec:OmegaH}, and so is the bound 
$$ |\mathcal{U}[F_{H}]|  \lesssim \frac{1}{\Re(\lambda) |\lambda|^{1/2}}. $$
It follows that the matrix $M_H$ defined in \eqref{defi_MH} satisfies $|i k M_H| < 1$  under \eqref{k:hyp}-\eqref{lambda:hyp}, and the the series defining $\hat{\omega}$ converges. The estimate of Proposition \ref{pro:main2HNS} on $\hat{\omega}$ can be deduced like the one of $\omega_{inhom}$ in \eqref{estim_omega_inhom} (the better power of $k$ comes from the fact that we use the weight $U''_{s}$ instead of the modified $U''_{s,k}$).  This concludes the proof. 

\vspace{5 mm}

\noindent \textbf{Acknowledgements:} SI was partially supported by NSF Grant DMS-1802940 when this project was initiated, and also acknowledges a startup grant from UC Davis. D.G-V. acknowledges the support of SingFlows project, grant
ANR-18- CE40-0027 of the French National Research Agency (ANR) and of the Institut Universitaire de France.


%
%
%

\end{document}